\documentclass[a4paper,11pt]{amsart}
\usepackage{amsmath, amssymb, amsthm, xcolor,mathtools}
\usepackage{verbatim}
\usepackage{graphicx}
\usepackage{enumerate}
\usepackage[colorlinks=true, citecolor=blue, linkcolor=blue, bookmarks=true]{hyperref}
\usepackage{tikz}
\usepackage{tikz-cd}
\usepackage{multirow, tabu}
\usetikzlibrary{intersections}

\newtheorem{theorem}{Theorem}
\newtheorem{lemma}[theorem]{Lemma}
\newtheorem{corollary}[theorem]{Corollary}
\newtheorem{proposition}[theorem]{Proposition}
\newtheorem*{theorem*}{Theorem}

\theoremstyle{definition}
\newtheorem{definition}[theorem]{Definition}
\newtheorem{remark}[theorem]{Remark}

\newtheorem{question}[theorem]{Question}

\numberwithin{equation}{section}

\DeclareMathOperator{\id}{id}
\DeclareMathOperator{\Ad}{Ad}

\newcommand{\HM}{\mathrm{HM}}
\newcommand{\OZ}{\mathrm{OZ}}

\newcommand{\dimnuc}{\mathrm{dim}_\mathrm{nuc}}

\newcommand{\N}{\mathbb{N}}
\newcommand{\M}{\mathbb{M}}
\newcommand{\C}{\mathbb{C}}
\newcommand{\K}{\mathbb{K}}
\renewcommand{\O}{\mathcal{O}}

\newcommand{\tr}{\mathrm{tr}}
\allowdisplaybreaks

\begin{document}
	\title[Nuclear dimension of extensions]{Nuclear dimension of extensions of commutative $\mathrm{C}^*$-algebras by Kirchberg algebras}

\author[]{Samuel Evington}
\address{Samuel Evington, Mathematical Institute, University of M\"unster, Einstein-strasse 62, 48149 M\"unster, Germany}
\email{evington@uni-muenster.de}

\author[]{Abraham C.S. Ng}
\address{Abraham C.S. Ng, School of Mathematics and Statistics, University of Sydney, New South Wales 2006, Australia}
\email{abraham.ng@sydney.edu.au}

\author[]{Aidan Sims}
\address{Aidan Sims, School of Mathematics and Applied Statistics, University of Wollongong, New South Wales 2522, Australia}
\email{asims@uow.edu.au}

\author[]{Stuart White}
\address{Stuart White, Mathematical Institute, University of Oxford,
  Oxford, OX2 6GG, United Kingdom}
\email{stuart.white@maths.ox.ac.uk}

\thanks{This work was supported by:
Deutsche Forschungsgemeinschaft (DFG, German Research Foundation) – Project-ID 427320536 – SFB 1442 [Evington]; 
Germany's Excellence Strategy EXC 2044 390685587  Mathematics M{\"u}nster: Dynamics–Geometry–Structure [Evington];  
ERC Advanced Grant 834267 - AMAREC [Evington];
Engineering and Physical Sciences Research Council (EP/X026647/1) [White]; ARC Discovery Project DP220101631 [Ng and Sims].  For the purpose of open access, the authors have applied a Creative Commons attribution (CC BY) licence to any Author Accepted Manuscript version arising from this submission.
}

\begin{abstract}
  We compute the nuclear dimension of extensions of $\mathrm{C}^*$-algebras involving commutative unital quotients and stable Kirchberg ideals. We identify the finite directed graphs whose $\mathrm{C}^*$-algebras are covered by this theorem.
\end{abstract}

\maketitle

\section{Introduction} \label{sec:intro}
\renewcommand{\thetheorem}{\Alph{theorem}}

By the Gelfand--Naimark theorem, every abelian $\mathrm{C}^*$-algebra arises as the algebra $C_0(X)$ of continuous functions that vanish at infinity on some locally compact Hausdorff topological space $X$. Thus $\mathrm{C}^*$-algebras provide a framework for non-commutative topology. Nuclear dimension is a non-commutative generalisation of covering dimension for topological spaces to the setting of $\mathrm{C}^*$-algebras due to Winter and Zacharias (\cite{WZ10}).  By Ostrand's theorem (\cite{Os71}), the covering dimension of a normal topological space $X$ is the smallest $n$ such that every locally finite open cover of $X$ has a refinement that can be $(n+1)$-coloured with sets of the same colour pairwise disjoint. The nuclear dimension of a $\mathrm{C}^*$-algebra is defined by `colouring' the finite-dimensional approximations from the Choi--Effros--Kirchberg characterisation of nuclearity (\cite{ChoiEff76,Kir77}) with approximants of the same colour mutually orthogonal. In particular, $\mathrm{C}^*$-algebras of finite nuclear dimension are nuclear; see Definition~\ref{def:DimNuc} below. For unital or separable abelian $\mathrm{C}^*$-algebras, nuclear dimension matches covering dimension: the classical proof that abelian $\mathrm{C}^*$-algebras satisfy the completely positive approximation property shows that $\dimnuc(C_0(X))\leq \dim X$; the reverse inequality is obtained by taking a good finite nuclear dimension approximation for a partition of unity (see \cite{Cas21}, for example). 

Since its introduction, nuclear dimension has been intimately linked with the Elliott classification programme for simple separable nuclear $\mathrm{C}^*$-algebras. The additional control given by an $(n+1)$-coloured completely positive approximations, as opposed to standard completely positive approximations, has structural consequences (see, for example, \cite{Wi12,Ro11}). Moreover, the original proof of the stably finite part of the unital classification theorem  (\cite{GLN20a,GLN20b, EGLN,TWW17}) used finite nuclear dimension as the fundamental regularity hypothesis so as to use Winter's technique, powered by nuclear dimension, of classification by embeddings (\cite{Winter:AJM}). Conversely, classification has been essential in obtaining optimal values of the nuclear dimension both directly (by transferring explicit finite dimensional approximations from families of model $\mathrm{C}^*$-algebras to the general setting in \cite{WZ10,En15,RST15,RSS15}) and more indirectly (through the use of classification theorems for maps to construct finitely colourable approximations as in \cite{MS14,SWW15,BBSTWW,CETWW,CE,BGSW19,Gla20,GT22,Ev22}).

This paper is concerned with the computation of the nuclear dimension of an extension
\begin{equation}
 0 \rightarrow J \rightarrow E \rightarrow B\rightarrow 0
\end{equation} 
in terms of $J$ and $B$.  Finite nuclear dimension is preserved by extensions:
\begin{equation}\label{Intro.Bounds}
\begin{split}
\max(\dimnuc (J),\dimnuc(B)) &\le \dimnuc(E) \\
        &\leq \dimnuc(J) +\dimnuc(B)+1.
\end{split}
\end{equation}
Indeed, this result was one of the primary motivations for the introduction of nuclear dimension in \cite{WZ10} (particularly in contrast with the earlier concept of decomposition rank from \cite{KW04}).\footnote{As $\mathrm{C}^*$-algebras of finite decomposition rank are quasidiagonal and hence stably finite, the Toeplitz algebra below has infinite decomposition rank.} The upper bound above is obtained by taking an approximate unit $(h_n)$ for $J$ that is quasicentral for $E$, and then using the approximation 
\begin{equation}\label{Intro.Bounds2}
a\approx h_n^{1/2}ah_n^{1/2}+(1-h_n)^{1/2}a(1-h_n)^{1/2},\quad a\in E.
\end{equation}
The point is that the first term $h_n^{1/2}xh^{1/2}_n$ lies in $J$, so can be approximated by a nuclear-dimension approximation for $J$, and for large $n$ one can obtain an approximation to the second term from a nuclear-dimension approximation for the quotient $B$ (this is the proof of \cite[Proposition 2.9]{WZ10}).\footnote{The number of colours used in an approximation is the nuclear dimension plus one, so this sketch explains the upper bound on $\dimnuc(E)$ in \eqref{Intro.Bounds}: the number of colours needed to approximate $E$ is at most the sum of the number of colours used to approximate the ideal and the quotient.}

In their original paper \cite{WZ10}, Winter and Zacharias highlighted the case of the Toeplitz algebra $\mathcal T$. This arises as an extension 
\begin{equation}
0\rightarrow\mathbb K\rightarrow \mathcal T\rightarrow C(\mathbb T)\rightarrow 0,
\end{equation}
and so\footnote{The compacts $\mathbb K$ have nuclear dimension $0$. Indeed, a $\mathrm{C}^*$-algebra has nuclear dimension $0$ if and only if it is approximately finite dimensional (see \cite[Remark 2.2(iii)]{WZ10}).}  has nuclear dimension $1$ or $2$. Despite the attention of experts, it took 10 years until Brake and Winter obtained the exact value of $1$ in \cite{BW19}. In a nutshell, Brake and Winter found a clever way to approximate suitable terms $(1-h_n)^{1/2}a(1-h_n)^{1/2}$ in \eqref{Intro.Bounds2} with the expected two colours (from the dimension of the quotient $C(\mathbb T)$) so that \emph{one colour lies orthogonal to $h_n$}. This orthogonality means that approximations to $h^{1/2}_n a h^{1/2}_n$ can be assigned this same colour; since only one colour is needed to approximate $h^{1/2}_n a h^{1/2}_n$ due to $0$-dimensionality of the ideal, the upshot is that only two colours are needed in total. Essentially, Brake and Winter succeed in reusing one of the two colours needed to approximate the quotient to approximate the ideal (we will discuss this strategy in a bit more detail in Section~\ref{sec:outline}).

The Brake--Winter result sparked the following question.

\begin{question}\label{Question.Main}
Let 
\begin{equation}
 0 \rightarrow J \rightarrow E \rightarrow B\rightarrow 0
\end{equation}
be an extension of $\mathrm{C}^*$-algebras. Is
\begin{equation}
\dimnuc(E)=\max(\dimnuc(J),\dimnuc(B))?
\end{equation}
\end{question}

In a small number of situations, a positive answer follows from general structural properties of the extension:
\begin{itemize}
\item when both $J$ and $B$ are separable and commutative, as then so too is $E$,\footnote{This can be seen, for example, by identifying $E$ with the pullback of the commutative algebras $M(J)$ and $B$ with diagram \eqref{eqn:busby}.} and we may use the classical result \cite[Corollary 3.5.8]{Pears75} to compute the covering dimension of the spectrum of $E$;
\item when $J$ and $B$ are both AF, in which case the extension $E$ is AF (\cite{Brown81,E76}), and so has nuclear dimension zero;
\item when both $J$ and $B$ are nuclear and $\mathcal O_\infty$-stable, when so too is the extension (\cite[Section 3]{TW07} or \cite[Section 4]{Kir06}) which then has nuclear dimension $1$ (\cite{BGSW19}); 

\item when the extension is quasidiagonal in the sense that there exists an approximate unit $(p_n)$ for $J$ which is quasicentral for $E$ and consists of projections (see \cite[Proposition 6.3]{RST15});\footnote{The point is that, as $p_n$ is a projection, it is orthogonal to $1-p_n$, so approximations coming from the quotient can be combined with those coming from the ideal without increasing the number of colours needed.}
\item when the extension arises from a minimal unitisation of a $\mathrm{C}^*$-algebra (see \cite[Remark 2.11]{WZ10} and \cite[Proposition 3.11]{KW04}).
\end{itemize}
The first three cases are explained by ``homogeneity'' of the extension, by which we mean that $E$ belongs to a class of $\mathrm{C}^*$-algebras for which nuclear dimension is known, and that is closed under ideals, quotients and extensions.

Following Brake and Winter’s work on the Toeplitz algebra, further positive answers to Question \ref{Question.Main} were obtained.  These results are summarised in the following table (in which $B$ is assumed nuclear).
\begin{center}
\begin{tabu}{|[1.5pt]r|c|c|c|c|[1.5pt]}
\tabucline[1.5pt]{-}
Paper&$J$&$B$&Extension&$\dimnuc(E)$\\
\tabucline[1.5pt]{-}
\cite{BW19}&$\mathbb K$&$C(\mathbb T)$&Toeplitz&1\\
\cite{Gla20}&Stable AF&Kirchberg&Unital&1\\
\cite{Ev22}&Stable AF&$\mathcal O_\infty$-stable&Full&$1$\\
\cite{GT22}&$\mathbb K$&$C(X)$&Essential&$\max(\dim X,1)$\\
\tabucline[1.5pt]{-}
\end{tabu}
\end{center}

In all these results, the nuclear dimension of the ideal is strictly less than that of the quotient, and this is crucial in deploying the Brake--Winter strategy to obtain an optimal nuclear dimension estimate (albeit using different powerful classification results in different cases to achieve the necessary orthogonality). Our main theorem provides the first answer to Question~\ref{Question.Main} other than the situations listed above.

\begin{theorem}\label{thm:main}
Let $J$ be a stable Kirchberg algebra and $X$ a compact metric space. Let 
\begin{equation}\label{eqn:extension1}
 0 \rightarrow J \rightarrow E \xrightarrow{\pi} C(X) \rightarrow 0
\end{equation}
be an essential extension. Then $\dimnuc(E) = \max(1,\dim(X))$.
\end{theorem}

When $\dim X>1$, the quotient has strictly bigger dimension than the ideal, and our result fits in to the Brake--Winter framework.  
The main challenge comes when $\dim X=1$, so the ideal and quotient each require $2$ colours. In this situation the upper bound in~\eqref{Intro.Bounds} gives $\dimnuc(E)\leq 3$. To prove Theorem~\ref{thm:main}, we therefore need to reuse \emph{both} colours from the quotient in approximating the ideal. We do so by arranging for one colour from the quotient to be orthogonal to both colours from the ideal (very similarly to the $\dim X>1$ case) and simultaneously for one of the two colours of the ideal to be orthogonal to both the colours from the quotient. 

Our arguments are powered by uniqueness theorems for $\O_\infty$-stable $^*$-homomorphisms, developed recently by Gabe (\cite{GabeOinf}), building on the work of Kirchberg (\cite{Ki95}). These results allow us to transfer ideas of Gardner and Tikuisis for extensions of $C(X)$ by $\K$ (\cite{GT22}), which were based on Lin's uniqueness theorem for approximately multiplicative maps from commutative $\mathrm{C}^*$-algebras into matrices (\cite{Lin17}), to the purely infinite setting. In the case when $\dim X=1$, a further subtle use of $\O_\infty$-stable classification is required to engineer the additional orthogonality required in the approximation of the ideal.  

One of our motivations for this work -- and in particular for focusing on the case $\dim X=1$ -- is the computation of the nuclear dimension of graph $\mathrm{C}^*$-algebras. The first results in this direction (\cite{RST15}) used the technique developed by Winter and Zacharias (\cite{WZ10}) for Cuntz algebras and Gabe's argument that stable extensions of graph $\mathrm{C}^*$-algebras in which the quotient is AF are quasidiagonal extensions to establish that graph $\mathrm{C}^*$-algebras that are an extension of an AF quotient by a purely infinite ideal have nuclear dimension~1. Subsequent work (\cite{En15, RSS15}) used models built from graph $\mathrm{C}^*$-algebras to show that all simple Kirchberg algebras in the UCT class have nuclear dimension~1, motivating subsequent general calculations of the optimal value (\cite{BBSTWW, CE, CETWW}). Moreover, with the striking classification results for graph $\mathrm{C}^*$-algebras of (\cite{ET09, ERRS21}) and the connection between classifiability and low nuclear dimension for simple $\mathrm{C}^*$-algebras (\cite{CETWW,CE}), it is natural to wonder whether graph algebras likewise generically have small nuclear dimension. Of particular relevance to Question \ref{Question.Main} is that every finite graph $\mathrm{C}^*$-algebra has a composition series whose ideals and subquotients have nuclear dimension at most $1$. Indeed, $\mathrm{C}^*$-algebras associated to finite directed graphs can be built up from iterated extensions of
\begin{itemize}
\item AF-algebras,
\item $\mathrm{C}^*$-algebras stably isomorphic to $C(\mathbb T)$, and
\item Kirchberg algebras.
\end{itemize}
Therefore, if Question \ref{Question.Main} has a positive answer, then every finite graph $\mathrm{C}^*$-algebra has nuclear dimension at most $1$ (and hence by taking direct limits, so too does every graph $\mathrm{C}^*$-algebra).  

\begin{question}\label{qn:dimnuc(graphalg)}
Do all graph $\mathrm{C}^*$-algebras have nuclear dimension at most~$1$?
\end{question}

Our theorem is designed for the stably commutative quotient, Kirchberg ideal setting (and Theorem~\ref{thm:main} readily extends to allow for stably commutative quotients; see Corollary~\ref{cor:stable quotient}).  We end the paper by giving a geometric characterisation (Proposition~\ref{prop:which graph algebras}) of the graphs for which our methods apply. Our argument relies heavily on the explicit correspondence between features of a graph and the ideal lattice of its $\mathrm{C}^*$-algebra (\cite{aHR97, HS02, CS2017}). In recent work, Farout and Schafhauser (\cite{FS23}) show that the $\mathrm{C}^*$-algebra of any graph with condition~(K) is an extension of an $\mathcal O_\infty$-stable quotient by a stable AF-ideal, and so obtain an upper bound of $2$ in general, and identify geometrically when \cite{Ev22} can be used to reduce the dimension further to~$1$. Condition~(K) precludes the presence of abelian subquotients with nontrivial spectrum, so our result does not intersect theirs. Indeed, their situation, where the quotient is $\mathcal{O}_\infty$-stable and the ideal stably finite, is complementary to ours, where the roles are reversed.

This paper is structured as follows. 
In Section~\ref{sec:prelims}, we set up our notational conventions and collect some prerequisite results on Kirchberg algebras and extensions.
We open Section~\ref{sec:outline} with a high level outline of the proof of Theorem~\ref{thm:main}, and then in Subsections \ref{subsec:extending-orderzero}--\ref{subsec:dim1} we develop the technical machinery needed to implement it. 
Section~\ref{sec:mainthm} contains the proof of Theorem~\ref{thm:main}.
In Section~\ref{sec:graphs} we relate our result to graph $\mathrm{C}^*$-algebras.

\renewcommand{\thetheorem}{\arabic{theorem}}
\numberwithin{theorem}{section} 
\section{Preliminaries} \label{sec:prelims}

We write $\M_n$ for the $\mathrm{C}^*$-algebra of $n \times n$ matrices and $\K$ for $\mathrm{C}^*$-algebra of compact operators on $\ell^2(\N)$. The Cuntz algebra with countably infinitely many generators from Cuntz's seminal paper \cite{Cu77} is denoted $\mathcal O_\infty$.

Let $A$ be a $\mathrm{C}^*$-algebra. We write $A_+$ for the positive elements in $A$ and  $A^\sim$ for the (forced) unitisation of $A$ (so $A^\sim \cong A \oplus \C$ when $A$ is already unital). We often implicitly work in the unitisation in our algebraic manipulations. The multiplier algebra of $A$ is denoted by $M(A)$, with corona $Q(A) = M(A)/A$ and quotient map $q_A:M(A) \rightarrow Q(A)$.

In this paper, approximate units are always assumed to be increasing nets of positive contractions, and are almost always sequential (as the $\mathrm{C}^*$-algebras of interest are separable). Following the language of \cite{GT22}, a sequential approximate unit $(h_n)_{n=1}^\infty$ is said to be \emph{almost idempotent} if $h_{n+1}h_n = h_n$ for all $n \in \N$. Every $\sigma$-unital $\mathrm{C}^*$-algebra has an almost idempotent approximate unit by \cite[Corollary~II.4.2.5]{Bl06}.

Quasicentral approximate units will play a particularly important role in this paper. Recall that given an ideal $I\triangleleft A$ in a $\mathrm{C}^*$-algebra, a \emph{quasicentral approximate unit} $(h_\lambda)_{\lambda \in \Lambda}$ for $I$ relative to $A$ is an approximate unit for $I$ with $\|h_\lambda a-a h_\lambda\|\to 0$ for all $a\in A$. 
The existence of quasicentral approximate units is guaranteed by Arveson's result \cite[Theorem~1]{Arveson}, whose proof shows that the convex hull of any approximate unit contains a quasicentral approximate unit. Consequently, when $A$ is separable, every ideal $I\triangleleft A$ has a sequential quasicentral almost idempotent approximate unit; see \cite[Proposition~II.4.3.2]{Bl06}. 

We fix a free ultrafilter $\omega$ on $\N$ for the entirety of the paper. The \emph{ultraproduct} of a sequence of $\mathrm{C}^*$-algebras $A_n$ with respect to $\omega$ is the $\mathrm{C}^*$-algebra 
\begin{equation}
 \prod_{n\to\omega} A_n = \frac{\prod_{n\in\N} A_n}{\{(a_n)_{n=1}^\infty\colon  \lim_{n\to\omega} \|a_n\|=0\}}. 
\end{equation}
We write $[a_n]_{n=1}^\infty$ for the element of the ultraproduct represented by the bounded sequence $(a_n)_{n=1}^\infty \in \prod_{n \in \N} A_n$. We write $A_\omega$ for the \emph{ultrapower} of a $\mathrm{C}^*$-algebra $A$, i.e $A_\omega = \prod_{n\to\omega} A$.

We will make frequent use of the embedding $\iota_{11}\colon A \rightarrow \M_2(A)$ given by
\begin{equation}
a \mapsto \begin{pmatrix}
a & 0 \\ 0 & 0
\end{pmatrix}
\end{equation}
for a variety of $\mathrm{C}^*$-algebras $A$. 

Let $A$ and $B$ be $\mathrm{C}^*$-algebras. Recall that a completely positive (c.p.) map $\phi\colon A \rightarrow B$ is said to be \emph{order zero} if it preserves orthogonality, i.e.\ $\phi(a)\phi(b) = 0$ whenever $a,b \in A_+$ satisfy $ab = 0$. 

There is a bijection between completely positive and contractive (c.p.c.) order zero maps $\phi\colon A\to B$ and $^*$-homomorphisms $\Phi\colon C_0(0,1]\otimes A\to B$  established in \cite[Corollary~3.1]{WZ09}. It will be convenient to have a standard notation for this, so we will write $\Phi = \HM(\phi)$ and $\phi = \OZ(\Phi)$.  So $\HM(\phi)$ is the unique $^*$-homomorphism $C_0(0,1]\otimes A\to B$ satisfying
\begin{equation}\label{DefHM}
    \phi(a) = \HM(\phi)(\mathrm{id}_{(0,1]} \otimes a), \quad a \in A,
\end{equation}
and $\OZ(\Phi(a)) = \Phi(\id_{(0,1]} \otimes a)$.

For completeness, we recall the definition of nuclear dimension of a completely positive map which is first formalised in \cite{TW} (for a $^*$-homomorphism) based on the seminal paper \cite{WZ10} introducing nuclear dimension for $\mathrm{C}^*$-algebras. We refer the reader to \cite{WZ10,WZ09} for more details on nuclear dimension and order zero maps.

\begin{definition}[{c.f.\ \cite[Definition~2.2]{TW}}]\label{def:DimNuc}
A completely positive map $\eta\colon A \to B$ between $\mathrm{C}^*$-algebras is said to have \emph{nuclear dimension at most $n$} if, for any finite subset $\mathcal{F} \subseteq A$ and any $\epsilon > 0$, there exist a finite-dimensional $\mathrm{C}^*$-algebra $G$, a c.p.c.\ map $\psi\colon  A \rightarrow G$ and a c.p.\ map $\phi\colon G \rightarrow B$ such that 
\begin{equation}
\max_{a \in \mathcal{F}}\|\eta(a) - \phi(\psi(a))\| < \epsilon
\end{equation}
and $G$ can decomposed as $G = G^{(0)} \oplus \cdots \oplus G^{(n)}$ where each restriction $\phi|_{G^{(i)}}$ is c.p.c.\ and order zero. The \emph{nuclear dimension} of $\eta$ is defined to be the least such $n$ (or $\infty$ if no such $n$ exists). Slightly rewriting history, the  \emph{nuclear dimension} of a $\mathrm{C}^*$-algebra $A$ is the nuclear dimension of the identity homomorphism $\id_A$.
\end{definition}

Recall that a $\mathrm{C}^*$-algebra is called \emph{Kirchberg} when it  is simple, separable, nuclear, and purely infinite. We will make use of the fact that Kirchberg algebras are $\O_\infty$-stable (one of Kirchberg's Geneva theorems from \cite{Ki95}; see \cite[Theorem~3.14]{0infinitystability}) and have real rank zero (see \cite[Theorem~1]{Zh90}). The existence of the following embeddings is well-known to experts but we supply a proof for the sake of completeness.
\begin{proposition}\label{prop:embeddings}
    Let $A$ be a Kirchberg algebra.
    \begin{enumerate}
        \item[(i)] There exists an embedding $\M_n \rightarrow A$ for all $n \in \N$.
        \item[(ii)] For each compact metric space $K$, there exists an embedding $C(K) \rightarrow A$; if $A$ is unital the embedding can be made unital.
\end{enumerate}
\end{proposition}
\begin{proof}
    (i): The matrix algebra $\M_{n}$ embeds (not unitally) in $\O_\infty$ via the map $\phi\colon \M_n \rightarrow \O_\infty$ given by $\phi(e_{ij}) = s_is_j^*$, where $(e_{ij})_{i=1,j=1}^n$ are the canonical generators of $\M_n$ and $(s_i)_{i\in\N}$ are the canonical generators of $\O_\infty$.
     Since $A$ is a Kirchberg algebra, $A \cong A \otimes \O_\infty$ and $A$ contains a non-zero projection $p$ (as it has real rank zero). We can therefore define an embedding $\psi\colon \M_{n} \rightarrow A \otimes \O_\infty \cong A$ via $\psi(x) = p \otimes \phi(x)$.

    (ii): As in part (i) (taking $p = 1_A$ if $A$ is unital), it suffices to show that $C(K)$ embeds unitally in $\O_\infty$. By the Hausdorff--Alexandroff theorem (see \cite{Alexandroff27, Hausdorff27}), it suffices to consider $K = \{0,1\}^\N$, the Cantor space. Since $\mathbb{C}^2 \cong D_2 = \operatorname{span}\{s_1 s_1^*, (1 - s_1 s_1^*)\}$, a unital subalgebra of $\mathcal{O}_\infty$, injectivity of the minimal tensor product gives $C(K) \cong \bigotimes_{i \in \mathbb{N}} D_2 \subseteq \bigotimes_{i \in \mathbb{N}} \mathcal{O}_\infty \cong \mathcal{O}_\infty$, which is a unital embedding.
\end{proof}

We now turn to the theory of extensions of $\mathrm{C}^*$-algebras. Our standard reference is \cite[Chapter 15]{Bla86}.
An \emph{extension} of $\mathrm{C}^*$-algebras is a short exact sequence 
\begin{equation}\label{eq:ext1}
0 \rightarrow J \xrightarrow{j} E \xrightarrow{\pi} B \rightarrow 0.
\end{equation} 
We say that $E$ is an extension of $B$ by $J$. We typically identify $J$ with the ideal $j(J) \triangleleft E$ and $B$ with the quotient $E/J$. 
The extension \eqref{eq:ext1} is said to be:
\begin{itemize}
    \item 
\emph{unital} if $E$ is unital (which forces $B$ to be unital);
\item  \emph{essential} if $J$ is an essential ideal in $E$, i.e.\ every non-zero ideal of $E$ intersects $J$ non-trivially;
\item  \emph{trivial} if there is a splitting $B \rightarrow E$ that is a $^*$-homomorphism.
\end{itemize}

An extension as in~\eqref{eq:ext1} induces a canonical $^*$-homomorphism $m_E : E \to M(J)$ given by $m_E(a)x = ax$ for $a\in E$ and $x\in J$. The \emph{Busby invariant} of an extension (\cite{Bus68}) is the unique map $\zeta\colon B \rightarrow Q(J)$ into the corona algebra of $J$ such that the following diagram commutes:
\begin{equation}\label{eqn:busby}
\begin{tikzcd}
E \arrow[r, "\pi", shift left] \arrow[d, "m_E"]               & B \arrow[d, "\zeta"] \\
M(J)\arrow[r, "q_J"]   & Q(J).     
\end{tikzcd}
\end{equation}

The Busby invariant completely encodes an extension via a pullback construction (see \cite[Example~15.3.2]{Bla86}).  In particular, an extension is unital if and only if its Busby invariant is unital. It is essential if and only if its Busby invariant is injective (see \cite[Section~15.2]{Bla86}). 

Two extensions with Busby invariants $\zeta_1, \zeta_2\colon  B \rightarrow Q(J)$ are \emph{strongly unitarily equivalent} if there exists a unitary $u \in M(J)$ such that $\zeta_2(b) = q_J(u)\zeta_1(b)q_J(u)^*$ for all $b \in B$.

The \textit{direct sum} of two extensions with Busby invariants $\zeta_1, \zeta_2\colon  B \rightarrow Q(J)$ is the extension with Busby invariant $\zeta_1 \oplus \zeta_2\colon B \rightarrow \M_2(Q(J)) \cong Q(\M_2(J))$ given by
\begin{equation} \label{directsum}
 	b \mapsto \begin{pmatrix}
 		\zeta_1(b) & 0\\
 		0 & \zeta_2(b)
 	\end{pmatrix}.
\end{equation}
In general this is an extension of $B$ by $\M_2(J)$; however, when $J$ is stable there is a canonical identification $\M_2(J) \cong J$; see \cite[Section 15.6]{Bla86}.

An extension of a separable nuclear $\mathrm{C}^*$-algebra $B$ by a separable stable $\mathrm{C}^*$-algebra $J$ with Busby invariant $\zeta\colon B \rightarrow Q(J)$ is \emph{absorbing} if  $\zeta \oplus \sigma$ is strongly unitarily equivalent to $\zeta$ for all trivial extensions $\sigma$.

For unital extensions, we need to adapt the definition slightly because $\zeta$ is never strongly unitarily equivalent to $\zeta \oplus 0$ if $\zeta$ is unital.
A unital extension is \emph{unitally trivial} if it admits a splitting that is a unital $^*$-homomorphism. A unital extension with Busby invariant $\zeta\colon B \rightarrow Q(J)$ is \emph{unitally absorbing} if $\zeta \oplus \sigma$ is strongly unitarily equivalent to $\zeta$ for every unitally trivial extension~$\sigma$.

The following theorem is well-known to experts but we do not know an explicit reference for the form we need, and so we show how to deduce it from existing results.
\begin{theorem}\label{thm:absorb}
    Let $J$ be a stable Kirchberg algebra and $B$ a separable nuclear $\mathrm{C}^*$-algebra. Then every (unital) essential extension of $B$ by $J$ is (unitally) absorbing.  
\end{theorem}
\begin{proof}
    The unital version of this result follows from the general theory of Elliott and Kucerovsky on absorbing extensions by combining \cite[Theorem~6]{El01} with \cite[Theorem~17(ii)]{El01} (noting that, as $B$ is nuclear, an extension of $B$ by $J$ is `absorbing in the nuclear sense' as given by \cite[Theorem~6]{El01} precisely when it is absorbing). However, this special case was first proven by Kirchberg (see \cite[Theorem~6]{Ki95}).

    For the non-unital version, there are some additional subtleties first observed by Gabe in \cite{Gabe16}. 
    Let $\zeta\colon B \rightarrow Q(J)$ be the Busby invariant of a non-unital essential extension of $B$ by $J$. 
    Then $\zeta$ is injective and non-unital. 
    It follows that the map $\zeta'\colon B^\sim \rightarrow Q(J)$ given by $b+ \lambda 1_{B^\sim} \mapsto \zeta(b) + \lambda 1_{Q(J)}$ is also injective. Since $J$ is a stable Kirchberg algebra, $Q(J)$ is simple by \cite[Theorem~3.2]{Ro91}, so $\zeta$ is unitisably full in the sense of \cite[Definition~2.5]{Gabe16}. Moreover $J$ satisfies the corona factorisation property by \cite[Proposition~2.1]{Ng06}. Hence, by \cite[Theorem~2.6]{Gabe16}, $\zeta$ is nuclearly absorbing. Since $B$ is nuclear, $\zeta$ is absorbing.
\end{proof}

\section{Machinery of proof}\label{sec:outline}

In this section, we outline the global strategy for proving Theorem \ref{thm:main} and isolate the technical results needed. 
The starting point is a decomposition that has its roots in \cite{BW19}. 
We formulate this as Proposition~\ref{prop:BW}, and will refer to this proposition throughout the paper to fix notation. 

\begin{proposition}[The Brake--Winter decomposition]\label{prop:BW} ~\\
    Let 
    \begin{equation}0 \rightarrow J \xrightarrow{\iota} E \xrightarrow{\pi} C(X) \rightarrow 0
    \end{equation}
    be an extension with $\mu\colon C(X) \rightarrow E$ a c.p.c.\ splitting. Let $(h_n)_{n=1}^\infty$ be a quasicentral approximate unit for $J$ relative to $E$. 
    Set 
    \begin{align}
    &A_n = \overline{h_nJh_n},\ 
    B_n = \overline{(h_{n+1} - h_n)J(h_{n+1} - h_n)},\text{ and }\nonumber\\&C_n = \overline{(1 - h_{n+1})E(1 - h_{n+1})}.
    \end{align}
    Define three sequences of c.p.c.\ maps 
    \begin{equation}
        \begin{split}
        \alpha_n\colon  E \rightarrow A_n,\quad  a &\mapsto h_n^{1/2} a h_n^{1/2},\\
        \beta_n\colon   C(X) \rightarrow B_n,\quad f &\mapsto (h_{n+1} - h_n)^{1/2}\mu(f)(h_{n+1} - h_n)^{1/2},\\
        \gamma_n\colon  C(X) \rightarrow C_n,\quad f &\mapsto (1-h_{n+1})^{1/2} \mu(f) (1-h_{n+1})^{1/2}.
        \end{split}
    \end{equation}
    Let $\alpha\colon E \rightarrow \prod_{n\to\omega} A_n \subseteq E_\omega$, $\beta\colon C(X) \rightarrow \prod_{n\to\omega} B_n \subseteq E_\omega$ and $\gamma\colon C(X) \rightarrow \prod_{n\to\omega} C_n \subseteq E_\omega$ be the induced c.p.c.\ maps into the respective ultraproducts. 
    Then $\alpha$,$\beta$ and $\gamma$ are c.p.c.\ order zero maps, and we have 
    \begin{equation}\label{eqn:QC-approx-1}
	\alpha(a) + \beta (\pi(a)) + \gamma(\pi(a)) = a
\end{equation}
for all $a \in E$.
\end{proposition}
\begin{proof}
    Write $h = [h_n]_{n=1}^\infty \in E_\omega$. Since $(h_n)_{n=1}^\infty$ is quasicentral, $\alpha(a) = ha = ah$ for all $a \in E$. Hence, if $a,b \in E_+$ satisfy $ab=0$, then $\alpha(a)\alpha(b) = habh = 0$. Therefore, $\alpha$ is a c.p.c.\ order zero map.
    
    Write $\hat{h} = [h_{n+1}]_{n=1}^\infty \in E_\omega$. Since $(h_n)_{n=1}^\infty$ is an approximate unit for $J$, we have $ha = a = \hat{h}a$ for all $a \in J \subseteq E_\omega$. Hence, $(\hat{h}-h)a = 0 = (1-\hat{h})a$ for all $a \in J \subseteq E_\omega$. Since $(h_n)_{n=1}^\infty$ is quasicentral, $(\hat{h}-h)$ and $(1-\hat{h})$ commute with all elements of $E \subseteq E_\omega$.

    Let $f_1,f_2 \in C(X)_+$. Then $\mu(f_1)\mu(f_2) - \mu(f_1f_2) \in J$. Hence, $(\hat{h}-h)\mu(f_1)\mu(f_2) = (\hat{h}-h)\mu(f_1f_2)$ and $(1-\hat{h})\mu(f_1)\mu(f_2)= (1-\hat{h})\mu(f_1f_2)$. Therefore, if $f_1f_2 = 0$, then
    \begin{equation}
        \beta(f_1)\beta(f_2) = (\hat{h}-h)^2\mu(f_1)\mu(f_2) =  (\hat{h}-h)^2\mu(f_1f_2) = 0,
    \end{equation}
    and similarly $\gamma(f_1)\gamma(f_2) = 0$. Hence, $\beta$ and $\gamma$ are c.p.c.\ order zero maps.

    Let $a \in E$. Then $\pi(\mu(\pi(a))) = \pi(a)$. Hence, $\mu(\pi(a)) - a \in J$. Therefore,  $(\hat{h}-h)\mu(\pi(a)) = (\hat{h}-h)a$ and $(1-\hat{h})\mu(\pi(a)) = (1-\hat{h})a$.
    Hence
    \begin{align}
        \alpha(a) + \beta(\pi(a)) + \gamma (\pi(a)) &= ha + (\hat{h}-h)\mu(\pi(a)) + (1-\hat{h})\mu(\pi(a)) \nonumber \\
        &= ha + (\hat{h}-h)a + (1-\hat{h})a  \\
        &= a.\qedhere \nonumber    
    \end{align}
\end{proof}

The power of the Brake--Winter decomposition comes from the fact that when the quasicentral approximate unit $(h_n)_{n=1}^\infty$ is almost idempotent, the hereditary subalgebras $A_n$ and $C_n$ are orthogonal. This is important as a sum of order zero maps with orthogonal ranges is again order zero. This means that the colours used in nuclear dimension approximations for $\gamma_n$ can be re-used in nuclear dimension approximations for $\alpha_n$, so that 
\begin{equation}
\dimnuc(\alpha_n+\gamma_n \circ \pi)\leq \max(\dimnuc(\alpha_n),\dimnuc(\gamma_n \circ \pi)).
\end{equation}
Since the image of $\alpha_n$ lies in $J$, we have $\dimnuc(\alpha_n) \leq \dimnuc(J)$; since $\gamma_n \circ \pi$ factors through $C(X)$, the estimate $\dimnuc(\gamma_n \circ \pi) \leq \dim(X)$ is true in a limiting sense as $n \to \infty$.\footnote{The careful choice of wording is because $\gamma_n$ is only \emph{approximately} order zero.}  
Therefore, Question~\ref{Question.Main} can be answered affirmatively for a given extension by finding a suitable strategy for handling the $\beta_n$ term without increasing the estimate of $\max(\dimnuc(J),\dim(X))$.\footnote{Note that, if the extension is a direct sum or has a quasicentral approximate unit consisting of projections, then $\beta_n$ can be  taken to be zero.}  

The strategy for handling the $\beta_n$ terms in the proof of Theorem \ref{thm:main} is as follows:
After approximating $\gamma_n$ using a $(\dim(X)+1)$-coloured open cover of $X$, we shall extend the collection of open sets corresponding to the first colour to a Borel partition of $X$, which will be used to approximate $\beta_n$. This requires extending the domain of the approximately order zero c.p.c.\ maps $\beta_n$ to include (some) Borel functions that are not continuous. 

In Section \ref{subsec:extending-orderzero}, we use Gabe's $\O_\infty$-stable uniqueness theorem to prove a suitable extension theorem for the c.p.c.\ order zero map $\beta$ in the Brake--Winter decomposition under the hypothesis that  $\HM(\beta)$ is injective.\footnote{Recall from \eqref{DefHM} that $\HM(\beta)$ is the $^*$-homomorphism whose domain is a cone corresponding to the c.p.c.\ order zero map $\beta$.}
In Section \ref{subsec:GT}, we show that the quasicentral approximate unit $(h_n)_{n=1}^\infty$ can be chosen such that the corresponding Brake--Winter decomposition does have the property that $\HM(\beta)$ is injective. 
As discussed in the introduction, to prove Theorem \ref{thm:main} in the case $\dim(X)=1$ requires some extra work: we also need to approximate the maps $\alpha_n$ in such a way that one colour is orthogonal to $\beta_n$. This is the topic of Section \ref{subsec:dim1}.

\subsection{Extending order zero maps} \label{subsec:extending-orderzero}

Starting from the Brake--Winter decomposition (Proposition \ref{prop:BW}), the key idea in \cite{GT22} is to choose the quasicentral approximate unit $(h_n)_{n=1}^\infty$ so that the $B_n$ are matrix algebras and then use Lin's uniqueness theorem for approximately multiplicative maps into matrix algebras \cite[Theorem~2.10]{Lin17} to extend $\beta$ to a c.p.c.\ order zero map on the $\mathrm{C}^*$-algebra $B(X)$ of all bounded Borel functions on $X$. 

The appearance of matrix algebras and the applicability of Lin's uniqueness theorem is specific to the extensions considered in \cite{GT22} where the ideal is $\K$. To transfer this idea to purely infinite ideals, we would like to prove that a c.p.c.\ order zero map $\beta\colon C(X) \rightarrow \prod_{n\to\omega} B_n$ into an ultraproduct of Kirchberg algebras can be extended to a c.p.c.\ order zero map defined on the $\mathrm{C}^*$-algebra $B(X)$ of all bounded Borel functions on $X$. In fact, we prove something slightly weaker:  $\iota_{11} \circ \beta\colon C(X) \rightarrow \M_2(\prod_{n\to\omega} B_n)$ can be extended to any \emph{separable} abelian $\mathrm{C}^*$-algebra containing $C(X)$ provided $\HM(\beta)$ is injective. This is sufficient for the proof of Theorem \ref{thm:main}.

We begin by deriving the version of Gabe's uniqueness theorem (\cite[Theorem~B]{GabeOinf}) that we need to prove the extension theorem. 
\begin{proposition}\label{prop:unitaryequivhomgen}
    Let $B_n$ be a sequence of Kirchberg algebras. Write $B_\omega = \prod_{n\to\omega} B_n$.\footnote{This is the first occurrence of abuse of notation we will make use of in this paper: $B_\omega$ is just a shorthand for the ultraproduct of the $B_n$, not an ultrapower.}
    Let $E$ be a separable nuclear $\mathrm{C}^*$-algebra.  Suppose $\phi, \psi\colon C_0(0,1]\otimes E \rightarrow B_\omega$ are injective $^*$-homomorphisms.
    Then there exists a unitary $u$ in the minimal unitisation $\M_2(B_\omega)^\sim$ of $\M_2(B_\omega)$ such that $\Ad(u)\circ(\iota_{11}\circ\phi) = \iota_{11}\circ \psi.$
\end{proposition}
\begin{proof}
    Let $A = C_0(0,1]\otimes E$ for convenience. Since $E$ is separable and nuclear, $A$ is separable and nuclear.

    To apply Gabe's uniqueness theorem \cite[Theorem~B]{GabeOinf}, we first use the theory of separably inheritable properties (see \cite[Section II.8.5]{Bl06}) to replace $B_\omega$ with a separable subalgebra $D \subseteq B_\omega$ containing $\phi(A)$ and $\psi(A)$ that is both simple and $\O_\infty$-stable.
    
    Since each $B_n$ is simple and purely infinite, $B_\omega$ is simple by \cite[Remark~2.4]{Kir06}. Simplicity is separably inheritable by \cite[Theorem~II.8.5.6]{Bl06}.
    Since each $B_n$ is Kirchberg, we have $B_n \cong B_n \otimes \O_\infty$ for all $n \in \N$, so \cite[Proposition~1.12]{Schaf20} says that the ultraproduct is \emph{separably} $\O_\infty$-stable, meaning that for every separable subalgebra $C_1 \subseteq B_\omega$ there is a separable $\O_\infty$-stable subalgebra $C_2$ with $C_1 \subseteq C_2 \subseteq B_\omega$. Being separably $\O_\infty$-stable is separably inheritable as $\O_\infty$-stability is preserved by sequential inductive limits by \cite[Corollary~3.4]{TW07}. Being both simple and separably $\O_\infty$-stable is therefore separably inheritable by \cite[Theorem~II.8.5.3]{Bl06}. So the required simple separable $\O_\infty$-stable subalgebra  $D \subseteq B_\omega$  exists.
    
    View $\phi$ and $\psi$ as maps $A \rightarrow D$. Since $A$ is nuclear, $\phi$ and $\psi$ are nuclear. Since $D$ is $\O_\infty$-stable, $\phi$ and $\psi$ are strongly $\O_\infty$-stable by \cite[Proposition~4.5]{GabeOinf}. Since $D$ is simple, injectivity of $\phi$ and $\psi$ ensures that they are full.
    As $D$ is separable, it is $\sigma$-unital.
    Since $C_0(0,1]\otimes E$ is homotopic to $0$, it follows that $KK_{nuc}(\phi) = KK_{nuc}(\psi) = 0$. Therefore, by \cite[Theorem~B]{GabeOinf}, $\phi,\psi\colon A \to D$ are asymptotically Murray--von Neumann equivalent (in the sense of \cite[Definition~3.4]{Gabe20}) and hence so are $\phi,\psi\colon  A\to B_\omega$.
	
    By \cite[Proposition~3.10]{Gabe20}, $\iota_{11}\circ\phi,\iota_{11} \circ \psi \colon  A\to \M_2(B_\omega)$ are asymptotically unitarily equivalent via unitaries in $\M_2(B_\omega)^\sim$. By an application of Kirchberg's $\epsilon$-test (\cite[Lemma~A.4]{Kir06}), there exists $u \in \M_2(B_\omega)^\sim$ such that $\Ad(u)\circ(\iota_{11}\circ\phi) = \iota_{11}\circ \psi.$
\end{proof}

We deduce the promised result on extending c.p.c.\ order zero maps.

\begin{proposition}\label{prop:extend}
 Let $X$ be a compact metric space. Let $B_n$ be a sequence of Kirchberg algebras. Write $B_\omega = \prod_{n\to\omega}B_n$. 
 Let $\beta\colon C(X) \rightarrow B_\omega$ be a c.p.c.\ order zero map such that $\HM(\beta)$ is injective. Let $D$ be a separable abelian $\mathrm{C}^*$-algebra containing $C(X)$.
 Then $\iota_{11} \circ \beta\colon C(X) \rightarrow \M_2(B_\omega)$ can be extended to an order zero map $D \rightarrow \M_2(B_\omega)$.
\end{proposition}
\begin{proof}	
 View $D$ as $C(Z)$ for a compact metric space $Z$. By Proposition~\ref{prop:embeddings}(ii), $C([0,1] \times Z)$ embeds in $B_n$ for each $n \in \N$. Hence, there exists an embedding, $C([0,1] \times Z) \rightarrow B_\omega$.
 Let $\psi\colon C_0((0,1] \times Z) \rightarrow B_\omega$ be the restriction of this embedding to $C_0((0,1]\times Z)$.
 
 Let $\phi = \HM(\beta)\colon C_0((0,1] \times X) \rightarrow B_\omega$ be the $^*$-homomorphism corresponding to the c.p.c.\ order zero map $\beta$. By hypothesis, $\phi$ is injective. 
 
 Applying Proposition~\ref{prop:unitaryequivhomgen} with $E=C(X)$ to $\phi$ and $\psi$, we see that $\iota_{11} \circ \phi$ and $\iota_{11} \circ \psi|_{C_0((0,1] \times X)}$ are unitarily equivalent via a unitary $u \in \M_2(B_\omega)^\sim$. Then $\iota_{11}(\phi(f)) = u\iota_{11}(\psi(f))u^*$ for all $f \in C_0((0,1] \times X)$. Therefore $\mathrm{Ad}(u) \circ (\iota_{11} \circ \psi)$ is an extension of $\iota_{11} \circ \phi$ to $C_0((0,1] \times Z)$. Hence $g \mapsto \mathrm{Ad}(u) \circ \iota_{11}(\psi(\id_{(0,1]} \otimes g))$ is a c.p.c.\ order zero extension of $\iota_{11} \circ \beta$.
\end{proof}

\subsection{Ensuring injectivity}\label{subsec:GT}

The Brake--Winter decomposition  works for any quasicentral approximate unit $(h_n)_{n=1}^\infty$. In this section, we shall show that, for any essential extension of $C(X)$ by a stable Kirchberg algebra, an almost idempotent quasicentral approximate unit can be chosen such that the c.p.c.\ order zero map $\beta$ in the Brake--Winter decomposition (using the notation of Proposition~\ref{prop:BW}) has the property that $\HM(\beta)$ is injective. This is necessary in order to apply our order zero extension result (Proposition~\ref{prop:extend}).

The method of proof is to first construct a trivial extension where the result holds, and then use that essential extensions of separable nuclear $\mathrm{C}^*$-algebras by Kirchberg algebras absorb trivial representations (see Theorem~\ref{thm:absorb}). This technique has its roots in Gardner and Tikuisis' work on the nuclear dimension of generalised Toeplitz algebras \cite{GT22}. 
Our starting point is the following proposition, which is essentially a restatement of \cite[Lemma~3.1]{GT22}.
To simply notation, we write $\tr_\omega$ for the canonical limit trace on $\prod_{n \to \omega} \M_{m_n}$ given by $\tr_\omega([x_n]_{n=1}^\infty) = \lim_{n\to\omega} \tr_{m_n}(x_n)$ -- this is in fact the unique trace on $\prod_{n \to \omega} \M_{m_n}$ by \cite[Theorem~ 1.2]{NR16}.

\begin{proposition}\label{prop:GT}
Let $X$ be a compact metric space and let $\tau$ be a state on $C_0(0,1] \otimes C(X)$. 
There exist a unital $^*$-homomorphism $\psi_0\colon C(X) \rightarrow M(\K)$, an almost idempotent quasicentral approximate unit $(h_n^{(0)})_{n=1}^\infty$ for $\K$ relative to $\mathcal{T}_{\psi_0} \colon = \K + \psi_0(C(X))$, and a sequence $(m_n)_{n=1}^\infty$ in $\mathbb{N}$ such that
\begin{enumerate}
    \item 
$\overline{\psi}_0 = q_\K \circ \psi_0\colon C(X) \rightarrow Q(\K)$ is an injective $^*$-homomorphism, 
\item $\overline{(h_{n+1}^{(0)}-h_n^{(0)})\K(h_{n+1}^{(0)}-h_n^{(0)})} \cong \M_{m_n}$ for all $n \in \N$, and 
\item the c.p.c. order zero map given by
\begin{align}
	\Psi_0\colon C(X) &\rightarrow \prod_{n \to \omega} \M_{m_n}, \quad
		f	  \mapsto [(h_{n+1}^{(0)}-h_n^{(0)})\psi_0(f)]_{n=1}^\infty 
\end{align}	
satisfies 
\begin{equation}
    \tr_\omega(\HM(\Psi_0)(f)) = \tau(f),\quad f \in C_0(0,1] \otimes C(X).
\end{equation}
\end{enumerate}
\end{proposition}
\begin{proof}
    This is an ultraproduct reformulation of \cite[Lemma~3.1]{GT22}. The construction of $\psi_0$ and $(h_n^{(0)})_{n=1}^\infty$ are given there. We shall just explain the reformulation.
    
    Co-restricting $\psi_0$ to a map $C(X) \rightarrow \mathcal{T}_{\psi_0}$, we see that it defines a c.p.c.\ splitting of the extension $\K \rightarrow \mathcal{T}_{\psi_0} \rightarrow \overline{\psi}_0(C(X)) \cong C(X)$. Hence, $\Psi_0$ is just the $\beta$ map in the Brake--Winter decomposition of this extension with respect to $(h_n^{(0)})_{n=1}^\infty$. It is therefore is c.p.c.\ order zero by Proposition~\ref{prop:BW}.
    
    For $k \in \N$ and $g \in C(X)$,
    \begin{equation}
        \begin{split}
        \HM(\Psi_0)(\id_{(0,1]}^k \otimes g) &= \HM(\Psi_0)(\id_{(0,1]} \otimes g^{1/k})^k \\
        &= [((h_{n+1}^{(0)}-h_n^{(0)})\psi_0(g^{1/k}))^k]_{n=1}^\infty \\
        &=  [(h_{n+1}^{(0)}-h_n^{(0)})^{k/2}\psi_0(g)(h_{n+1}^{(0)}-h_n^{(0)})^{k/2}]_{n=1}^\infty.
        \end{split}
    \end{equation}
    Hence, $\HM(\Psi_0)$ is the map induced by the sequence of c.p.c.\ maps constructed in \cite[Lemma~3.1]{GT22}. Therefore, $\tr_\omega(\HM(\Psi_0)(f)) = \tau(f)$ for all $f \in C_0((0,1] \times X)$ by \cite[Lemma~3.1]{GT22}.
\end{proof}

From the above result for extensions by the compacts, we now deduce a similar result for extensions by stable Kirchberg algebras.

\begin{corollary}\label{cor:trivial-ext}
    Let $X$ be a compact metric space and $J$ a stable Kirchberg $\mathrm{C}^*$-algebra.
 There exist a unital $^*$-homomorphism $\psi\colon C(X) \rightarrow M(J)$ and an almost idempotent quasicentral approximate unit $(h_n)_{n=1}^\infty$ for $J$ relative to $J + \psi(C(X))$ such that
 \begin{enumerate}
     \item
$\overline{\psi} = q_J \circ \psi\colon C(X) \rightarrow Q(J)$ is an injective $^*$-homomorphism, and
\item the $^*$-homomorphism $\HM(\Psi)$ associated to the c.p.c. order zero map 
    \begin{equation}
        \begin{split}\label{eqn:Psi}
	\Psi\colon C(X) &\rightarrow \prod_{n \to \omega} \overline{(h_{n+1}-h_n)J(h_{n+1}-h_n)}\\
		f	  &\mapsto [(h_{n+1}-h_n)\psi(f)]_{n=1}^\infty 
        \end{split}
    \end{equation}	
    is injective.
\end{enumerate}
\end{corollary}
\begin{proof}
Fix a faithful state $\tau$ on $C_0((0,1] \times X)$. Let $\psi_0\colon C(X) \rightarrow M(\K)$ be the unital $^*$-homomorphism $\psi_0\colon C(X) \rightarrow M(\K)$ coming from Proposition~\ref{prop:GT} and let $(h_n^{(0)})_{n=1}^\infty$ be the corresponding almost idempotent quasicentral approximate unit. 
 
 The $^*$-homomorphism $\HM(\Psi_0)$ associated to the order zero map $\Psi_0$ constructed by Proposition~\ref{prop:GT} is injective: if $f \in C_0((0,1]\times X)$ satisfies $\HM(\Psi_0)(f) = 0$, then 
    \begin{equation}
         0 = \tr_\omega(\HM(\Psi_0)(f^*f)) = \tau(f^*f), 
    \end{equation}
    so $f = 0$, as $\tau$ is faithful.

  By, for example,  \cite[Theorem~1]{Zh90}, there exists a non-trivial projection $p \in J$. The corner $J_0 = pJp$ is a unital full hereditary subalgebra and so $J$ and $J_0$ are stably isomorphic by Brown's theorem \cite[Theorem~2.8]{Br77}. Hence, $J_0 \otimes \K \cong J \otimes \K \cong J$; we identify $J$ with $J_0 \otimes \K$.

We define $\psi\colon C(X) \rightarrow M(J)$ by $\psi(f) = 1_{J_0} \otimes \psi_0(f) \in J_0\otimes M(\mathbb{K}) \subseteq M(J_0\otimes \mathbb{K}) = M(J)$ and take $h_n = 1_{J_0} \otimes h_n^{(0)}$. Defining $\Psi$ as in \eqref{eqn:Psi}, we have 
$\Psi(f) = [1_{J_0} \otimes (h_{n+1}^{(0)}-h_n^{(0)})\psi_0(f)]_{n=1}^\infty$, so we can identify $\Psi$ with $1_{J_0} \otimes \Psi_0$. For $f\in C(X)$,
\begin{equation}
  \begin{split}
    (1_{J_0}\otimes \HM(\Psi_0))(\id_{(0,1]}\otimes f)&= (1_{J_0}\otimes \HM(\Psi_0))(\id_{(0,1]}\otimes f)\\&=1_{J_0}\otimes \Psi_0(f)\\&=\Psi(f).
  \end{split}
\end{equation}
So the defining equation \eqref{DefHM} shows that we can identify $\HM(\Psi)$ with $1_{J_0} \otimes \HM(\Psi_0)$, which is injective as $\HM(\Psi_0)$ is injective.
\end{proof}

Next, we take an arbitrary essential extension of $C(X)$ by a stable Kirchberg $\mathrm{C}^*$-algebra and add on the trivial extension that we have just constructed. This proposition is analogous to \cite[Corollary~3.2]{GT22}.

\begin{proposition}\label{prop:add-trivial}
Let $X$ be a compact metric space and $J$ be a stable Kirchberg algebra. Let $\phi\colon C(X) \rightarrow M(J)$ be a c.p.c.\ map such that $\overline{\phi}=q_J \circ \phi\colon C(X) \to Q(J)$ is an injective $^*$-isomorphism.

Then there exist a unital $^*$-homomorphism $\psi\colon C(X) \to M(J)$ and an almost idempotent quasicentral approximate unit $(h_n)_{n=1}^\infty$ for $\M_2(J)$ relative to $\M_2(J) + (\phi \oplus \psi)(C(X))$
such that
\begin{enumerate}
    \item $\overline{\psi}=q_J\circ\psi$ is an injective $^*$-homomorphism, and
    \item  the $^*$-homomorphism $\HM(\Gamma)$ corresponding to the c.p.c. order zero map 
    \begin{equation}
        \begin{split}
	\Gamma\colon C(X) &\rightarrow \prod_{n \to \omega} \overline{(h_{n+1}-h_n)\M_2(J)(h_{n+1}-h_n)}\\
		f	  &\mapsto [(h_{n+1}-h_n)(\phi(f) \oplus \psi(f))]_{n=1}^\infty 
        \end{split}
    \end{equation}
    is injective.
\end{enumerate}
\end{proposition}
\begin{proof}
Take $\psi$ and $(h_n^{\psi})_{n=1}^\infty$ to be the $^*$-homomorphism and the almost idempotent quasicentral approximate unit constructed in Corollary~\ref{cor:trivial-ext}, and define $\Psi$ as in \eqref{eqn:Psi} using the approximate unit  $(h_n^{\psi})_{n=1}^\infty$. 
 
Let $(h_n^{\phi})_{n=1}^\infty$ be any almost idempotent quasicentral approximate unit for $J$ relative to $J + \phi(C(X))$ and define $\Phi$ analogously to \eqref{eqn:Psi} with $\phi$ replacing $\psi$ and using the approximate unit $(h_n^{\phi})_{n=1}^\infty$. Note that $\Phi$ is c.p.c. order zero by Proposition~\ref{prop:BW}, as it can be viewed as the $\beta$ map of the Brake--Winter decomposition for the extension $J \rightarrow J + \phi(C(X)) \rightarrow \overline{\phi}(C(X)) \cong C(X)$ with respect to $(h_n^{\phi})_{n=1}^\infty$.  

Set $h_n = h_n^{\phi} \oplus h_n^{\psi}$ for $n \in \N$. Then $(h_n)_{n=1}^\infty$ is an almost idempotent quasicentral approximate unit for $\M_2(J)$ relative to $\M_2(J) + (\phi \oplus \psi)(C(X))$.
By construction, $\Gamma = \Phi \oplus \Psi$. Hence, $\HM(\Gamma) = \HM(\Phi) \oplus \HM(\Psi)$. By Corollary~\ref{cor:trivial-ext}, $\HM(\Psi)$ is injective. Hence, so is $\HM(\Gamma)$. 
\end{proof}

We can now prove the existence of a suitable Brake--Winter decomposition with $\HM(\beta)$ injective in the setting of Theorem~\ref{thm:main}.

\begin{proposition}\label{prop:beta-full}
Let $J$ be a stable Kirchberg algebra and $X$ a compact metric space. Let 
\begin{equation}\label{eqn:beta-full-1}
 0 \rightarrow J \rightarrow E \xrightarrow{\pi} C(X) \rightarrow 0
\end{equation}
be an essential extension. Then there exist a c.p.c.\ splitting $\mu\colon C(X)\rightarrow E$ and an almost idempotent quasicentral approximate unit $(h_n)_{n=1}^\infty$ for $J$ relative to $E$ such that the c.p.c.\ order zero map 
\begin{equation}
    \begin{split}
    \beta\colon C(X) &\rightarrow \prod_{n\to\omega} \overline{(h_{n+1}-h_n)J(h_{n+1}-h_n)}\\
    f &\mapsto [(h_{n+1}-h_n)\mu(f)]_{n=1}^\infty 
    \end{split}
\end{equation}
corresponds to an injective $^*$-homomorphism $\HM(\beta)$.
\end{proposition}
\begin{proof}
    Let $\overline{\phi}\colon C(X) \rightarrow Q(J)$ be the Busby invariant of the extension \eqref{eqn:beta-full-1}. 
    As the extension is essential, $\overline{\phi}$ is injective.
    As $C(X)$ is separable and nuclear, the Choi--Effros theorem \cite{ChoiEff76} yields a c.p.c.\ lift $\phi\colon C(X) \rightarrow M(J)$. The extension \eqref{eqn:beta-full-1} is isomorphic to the extension 
    \begin{equation}\label{eqn:beta-full-2}
        0 \rightarrow J \rightarrow J + \phi(C(X))  \rightarrow \overline{\phi}(C(X)) \cong C(X) \rightarrow 0
    \end{equation}
    as both extensions have Busby invariant $\overline{\phi}$.

    Let $\psi\colon C(X) \rightarrow M(J)$ be as in Proposition~\ref{prop:add-trivial}. The extension
    \begin{equation}\label{eqn:beta-full-3}
        0 \rightarrow \M_2(J) \rightarrow \M_2(J) + (\phi \oplus \psi)(C(X))  \rightarrow C(X) \rightarrow 0
    \end{equation}
    has a c.p.c.\ splitting given by the co-restriction of $\phi \oplus \psi$ to the subalgebra $\M_2(J) + (\phi \oplus \psi)(C(X))$.
    Proposition~\ref{prop:add-trivial} shows that there is an almost idempotent quasicentral approximate unit $(h_n)_{n=1}^\infty$ for this extension such that the $^*$-homomorphism $\HM(\beta)$ associated to the c.p.c.\ order zero map $\beta$ in the Brake--Winter decomposition is injective.
    
    Since $\overline{\phi}$ is essential and $\overline{\psi}$ is unitally trivial, Theorem~\ref{thm:absorb} ensures that $\overline{\phi}$ and $\overline{\phi} \oplus \overline{\psi}$ are strongly unitarily equivalent. Hence, there is an isomorphism $\kappa\colon E \rightarrow \M_2(J) + (\phi \oplus \psi)(C(X))$ such that the following diagram commutes:\footnote{Concretely, $\kappa$ has the form $\lambda \circ \Ad(u) \circ m_E$, where $\lambda\colon M(J)\rightarrow \M_2(M(J))$ is a canonical isomorphism (using that $J$ is stable), $m_E\colon E\rightarrow M(J)$ is the canonical embedding (using that $J$ is essential), and $u \in M(J)$ is the unitary implementing the strong unitary equivalence.}
    \begin{equation}
    \begin{tikzcd}
 J \arrow[r, "\iota", ] \arrow[d, "\kappa|_J"] & E \arrow[r, "\pi", ] \arrow[d, "\kappa"]  & C(X) \arrow[d, "\id_{C(X)}"] \\
 \M_2(J) \arrow[r, ""] &  \M_2(J) + (\phi \oplus \psi)(C(X)) \arrow[r, ""]   & C(X).      
\end{tikzcd} 
\end{equation}
Pulling back the almost idempotent quasicentral approximate unit and the c.p.c.\ splitting for the extension \eqref{eqn:beta-full-3} completes the proof.
\end{proof}

\subsection{Approximating the Kirchberg ideal}\label{subsec:dim1}

Recall that the strategy for proving Theorem \ref{thm:main} is to re-use \emph{one} of the colours from the nuclear dimension approximation of $\gamma_n$ to approximate $\beta_n$ (using the notation of Proposition~\ref{prop:BW}). The modified map will now have range in $\mathrm{C}^*(B_n \cup C_n)$, so is no longer orthogonal to $A_n$. This is not a problem if $\dim(X) > 1$, as we can reuse the \emph{other} colours from the approximation for $\gamma_n$ to approximate $\alpha_n$.  

However, when $\dim(X) \leq 1$,  we need to approximate $\alpha_n$ in such a way that one of the two colours used is orthogonal to both $B_n$ and $C_n$. That is the goal of this subsection. For technical reasons, it is easier to work with $\iota_{11} \circ \alpha_n\colon E \rightarrow \M_2(A_n)$ rather than $\alpha_n$.

We begin by using Gabe's uniqueness theorem \cite[Theorem~B]{GabeOinf} and Voiculescu's theorem on the homotopy invariance of quasidiagonality \cite{Vo91} to construct finite-dimensional approximations to certain injective $^*$-homomor\-phisms out of cones.   This is in the spirit of the nuclear dimension approximations for Kirchberg algebras, and more generally nuclear $\mathcal O_\infty$-stable algebras from \cite{BBSTWW,BGSW19}.

\begin{proposition}\label{prop:NDzeroFull*hom}
    Let $D_n$ be a sequence of Kirchberg algebras, and write $D_\omega$ for their ultraproduct. Let $E$ be a separable nuclear $\mathrm{C}^*$-algebra.
    Suppose that $\Phi_n\colon C_0(0,1] \otimes E \rightarrow D_n$ is a uniformly bounded sequence of maps such that the induced map $\Phi\colon  C_0(0,1] \otimes E \rightarrow D_\omega$ is an injective $^*$-homomorphism.
    Then there exist a sequence of c.p.c.\ maps $\eta_n\colon C_0(0,1] \otimes E \rightarrow \M_{r_n}$ and a sequence of $^*$-homomorphisms $\xi_n\colon \M_{r_n} \rightarrow \M_2(D_n)$ such that
    \begin{equation}
        \lim_{n\to\omega} \| \iota_{11} \circ \Phi_n(f) - \xi_n \circ \eta_n (f) \| = 0
    \end{equation}
    for all $f \in C_0(0,1] \otimes E$. 
\end{proposition}
\begin{proof}
   We construct a sequence of maps $\Psi_n\colon C_0(0,1] \otimes E \rightarrow D_n$ that approximately factor through $\M_{r_n}$, then use Gabe's uniqueness theorem to relate the maps $\Phi_n$ to the maps $\Psi_n$.

   The $\mathrm{C}^*$-algebra $C_0(0,1] \otimes E$ is quasidiagonal by \cite{Vo91} and separable since $E$ is separable, so  there exists a sequence of c.p.c.\ maps $\eta_n\colon C_0(0,1] \otimes E \rightarrow \M_{r_n}$ which are approximately multiplicative and approximately isometric. The induced map $\eta\colon C_0(0,1] \otimes E \rightarrow \prod_{n\to\omega} \M_{r_n}$ is thus an isometric $^*$-homomorphism.

   As each $D_n$ is Kirchberg, there exist embeddings $\xi_n'\colon \M_{r_n} \rightarrow D_n$ by Proposition~\ref{prop:embeddings}(i). Set $\Psi_n = \xi_n' \circ \eta_n\colon C_0(0,1] \otimes E \rightarrow D_n$ for each $n \in \N$. Since the $\xi_n'$ are isometric and the $\eta_n$ are approximately isometric, the induced map $\Psi\colon C_0(0,1] \otimes E \rightarrow D_\omega$ is an injective $^*$-homomorphism. 

   Since both $\Phi$ and $\Psi$ satisfy the hypotheses of Proposition~\ref{prop:unitaryequivhomgen}, there exists a unitary $u \in \M_2(D_\omega)^\sim$ such that $\Ad(u) \circ \iota_{11} \circ \Psi = \iota_{11} \circ \Phi$. 

   We can lift the unitary $u \in \M_2(D_\omega)^\sim$ to a representative sequence of unitaries $(u_n)_{n=1}^\infty$, where $u_n \in \M_2(D_n)^\sim$. 
   Define $^*$-homomorphisms $\xi_n = \Ad(u_n) \circ \iota_{11} \circ \xi_n' \colon \M_{r_n} \rightarrow \M_2(D_n)$.
   
   Now fix $f \in C_0(0,1] \otimes E$. Then  $(\xi_n(\eta_n(f)))_{n=1}^\infty$ and $(\iota_{11}(\Phi_n(f)))_{n=1}^\infty$ are representative sequences for $\Ad(u) \circ \iota_{11} \circ \Psi(f)$ and $\iota_{11} \circ \Phi (f)$, respectively. Since $\Ad(u) \circ \iota_{11} \circ \Psi(f) = \iota_{11} \circ \Phi (f)$, we have
      \begin{equation}
        \lim_{n\to\omega} \| \iota_{11} \circ \Phi_n(f) -\xi_n \circ \eta_n (f) \| = 0. \qedhere
    \end{equation} 
\end{proof}

We now construct the desired two-colour approximations to the maps $\iota_{11} \circ \alpha_n$. The idea, whose roots are in \cite{BBSTWW}, is to split $\alpha\colon E \rightarrow \prod_{n\to\omega} A_n$ into two pieces using a positive contraction $k$ with spectrum $[0,1]$ that commutes with the image of $\alpha$. We then approximate each piece using Proposition~\ref{prop:NDzeroFull*hom}. 

\begin{proposition}\label{prop:Approximate-Alpha}
    Let $E$ be a separable nuclear $\mathrm{C}^*$-algebra and let $J \triangleleft E$ be an essential ideal. Let $(h_n)_{n=1}^\infty$ be a quasicentral approximate unit for $J$ relative to $E$. Let $\alpha_n\colon E \rightarrow J$ be the approximately order zero c.p.c.\ map given by $a \mapsto h_n^{1/2}ah_n^{1/2}$.

    Suppose that $J$ is a Kirchberg algebra. Then there exist c.p.c.\ maps
    \begin{equation}
    \begin{split}
        \eta_n^{(0)}&\colon E \rightarrow \M_{r_n^{(0)}}, \quad
        \xi_n^{(0)}\colon \M_{r_n^{(0)}} \rightarrow \M_2(\overline{h_{n-1}Jh_{n-1}}),\\
        \eta_n^{(1)}&\colon E \rightarrow \M_{r_n^{(1)}}, \quad 
        \xi_n^{(1)}\colon \M_{r_n^{(1)}} \rightarrow \M_2(\overline{h_{n}Jh_{n}}),
    \end{split}
    \end{equation}
    such that $\xi_n^{(0)}$ and $\xi_n^{(1)}$ are $^*$-homomorphisms and 
    \begin{equation}\label{eq:approximates alpha}
        \lim_{n\to\omega} \big\| \iota_{11} \circ \alpha_n(a) - \big(\xi^{(0)}_n \circ \eta^{(0)}_n (a) + \xi^{(1)}_n \circ \eta^{(1)}_n (a)\big)\big\| = 0
    \end{equation}
    for all $a \in E$.
\end{proposition}
\begin{proof}
Write $h = [h_n]_{n=1}^\infty \in J_\omega$ and $\bar{h} = [h_{n-1}]_{n=1}^\infty \in J_\omega$. Since $(h_n)_{n=1}^\infty$ is quasicentral, the c.p.c.\ map $\alpha\colon E \rightarrow J_\omega$ induced by the sequence $(\alpha_n)_{n=1}^\infty$ is order zero and satisfies $\alpha(a) = ha$ for all $a \in E$. 

Since $J$ is $\O_\infty$-stable, there is a unital embedding $\O_\infty \rightarrow M(J)_\omega \cap J'$ by \cite[Theorem~7.2.6]{Rordam2002}.
Let $S$ be the separable $\mathrm{C}^*$-subalgebra of $J_\omega$ generated by $J \cup \bar{h}E \cup hE \cup \{h, \bar{h}\}$.   
Since $S \subseteq  J_\omega$ is separable, there is a unital embedding $\iota_{\O_\infty}\colon \O_\infty \rightarrow M(J)_\omega \cap S'$ by Kirchberg's $\epsilon$-test (\cite[Lemma~A.4]{Kir06}). 
Let  $\iota_{\O_\infty \otimes S}\colon  \O_\infty \otimes S \rightarrow J_\omega$ be the  $^*$-homomorphism given by $t \otimes s \mapsto  \iota_{\O_\infty}(t)s$. Then $\iota_{\O_\infty \otimes S}$ is injective because every ideal of $\O_\infty \otimes S$ is of the form $\O_\infty \otimes I$ for some $I \triangleleft S$ and $\iota_{\O_\infty \otimes S}(1 \otimes s) = s$ for all $s \in S$.

Let $k' \in \O_\infty$ be a positive contraction with spectrum $[0,1]$ and set $k = \iota_{\O_\infty}(k')$. Then $k$ has spectrum $[0,1]$ as $\iota_{\O_\infty}$ is a unital embedding.
We decompose the map $\alpha\colon E \rightarrow J_\omega$ as $\alpha = \alpha^{(0)} + \alpha^{(1)}$, where $\alpha^{(0)},\alpha^{(1)}\colon E \rightarrow J_\omega$ are given by
\begin{align}
	\alpha^{(0)}(a) &= k\bar{h}a,\text{ and} \label{eqn:oz1}\\
	\alpha^{(1)}(a) &= (1-k)\bar{h}a + (h-\bar{h})a\label{eqn:alpha1}
\end{align}
for $a \in E$.

Let $a \in E$. Since $k$ commutes with both $\bar{h}$ and $\bar{h}a$, and $\bar{h}$ commutes with $a$ as $(h_n)_{n=1}^\infty$ is quasicentral,
\begin{equation}
    k\bar{h}a = \bar{h}a k = a\bar{h}k = ak\bar{h}.
\end{equation}
Since $k$ and $\bar{h}$ commute, it follows that
\begin{align}
    \alpha^{(0)}(a) &= ak\bar{h} \label{eqn:oz2}\\
                    &= \bar{h}^{1/2}k^{1/2}ak^{1/2}\bar{h}^{1/2}. \label{eqn:cpc1}
\end{align}
Using \eqref{eqn:cpc1}, we see that $\alpha^{(0)}$ is a c.p.c.\ map with image in the hereditary subalgebra $\prod_{n\to\omega} \overline{h_{n-1}Jh_{n-1}}$. By \eqref{eqn:oz1}~and~\eqref{eqn:oz2}, if $a,b \in E$ satisfy $ab=0$, then $\alpha^{(0)}(a)\alpha^{(0)}(b)=0$. Hence, $\alpha^{(0)}$ is c.p.c.\ order zero.

Similarly, $\alpha^{(1)}$ is a c.p.c.\ order zero map with image in the hereditary subalgebra $\prod_{n\to\omega} \overline{h_{n}Jh_{n}}$. The equations corresponding to \eqref{eqn:oz2}~and~\eqref{eqn:cpc1} in this case are 
\begin{align}
    \alpha^{(1)}(a) &= a(1-k)\bar{h} + a(h-\bar{h}) \\
                    &= \bar{h}^{1/2}(1-k)^{1/2}a(1-k)^{1/2}\bar{h}^{1/2} + (h-\bar{h})^{1/2}a(h-\bar{h})^{1/2}.
\end{align}

We claim that $\HM(\alpha^{(0)})$ is injective. 
Since $J$ is an essential ideal of $E$, the ideal $C_0(0,1] \otimes J$ is essential in $C_0(0,1] \otimes E$.\footnote{Suppose $f \in C_0(0,1] \otimes E$ is non-zero. Then $f(t_0) \neq 0$ for some $t_0 \in (0,1]$, to there is $a \in J$ with $f(t_0)a \neq 0$. Take $g(t) = ta$. Then $g \in  C_0(0,1] \otimes J$ and $fg \neq 0$.}
Hence, it suffices to show that $\HM(\alpha^{(0)})|_{C_0(0,1] \otimes J}$ is injective.\footnote{If $\mathrm{Ker}(\HM(\alpha^{(0)})) \not= \{0\}$, then $\mathrm{Ker}(\HM(\alpha^{(0)})) \cap (C_0(0,1] \otimes J) \not= \{0\}$.}
Since $(h_{n-1})_{n=1}^\infty$ is an approximate unit for $J$, we have $\alpha^{(0)}(a) = ka$ for all $a \in J$. Hence, $\HM(\alpha^{(0)})(f \otimes a) = f(k)a$ for all $f \in C_0(0,1]$ and $a \in J$ (as the right hand side defines a $^*$-homomorphism satisfying \eqref{DefHM}). Therefore, $\HM(\alpha^{(0)})(f \otimes a) =\iota_{\O_\infty \otimes S}(f(k') \otimes a)$ for all $f \in C_0(0,1]$ and $a \in J$. The map $C[0,1] \rightarrow \O_\infty$ given by  $f \mapsto f(k')$ is injective as $k'$ has spectrum $[0,1]$. Hence, by injectivity of the minimal tensor product, the map $C[0,1] \otimes S\to \O_\infty\otimes S$ given by $f \otimes a \mapsto f(k') \otimes a$ is injective.
Since $\iota_{\O_\infty \otimes S}$ is injective, $\HM(\alpha^{(0)})|_{C_0(0,1] \otimes J}$ is injective, proving the claim.

A similar argument shows that $\HM(\alpha^{(1)})$ is injective. Indeed, it suffices to show that $\HM(\alpha^{(1)})|_{C_0(0,1] \otimes J}$ is injective. Since  $(h_{n})_{n=1}^\infty$ and  $(h_{n-1})_{n=1}^\infty$ are quasicentral approximate units for $J$ relative to $E$, we have $\alpha^{(1)}(a) = (1-k)a$ for all $a \in J$. Hence, $\HM(\alpha^{(1)})(f \otimes a) =\iota_{\O_\infty \otimes S}(f(1-k') \otimes a)$ for all $f \in C_0(0,1]$ and $a \in J$. As $1-k'$ is a positive contraction with full spectrum,  $\HM(\alpha^{(1)})|_{C_0(0,1] \otimes J}$ is injective. Therefore, $\HM(\alpha^{(1)})$ is injective.

Since $C_0(0,1] \otimes E$ is separable and nuclear, the Choi--Effros theorem \cite{ChoiEff76} implies that we can lift $\HM(\alpha^{(0)})$ to a sequence of c.p.c.\ maps $\widehat{\alpha}_n^{(0)}\colon C_0(0,1] \otimes E \rightarrow \overline{h_{n-1}Jh_{n-1}}$ and that we can lift $\HM(\alpha^{(1)})$ to a sequence of c.p.c.\ maps $\widehat{\alpha}_n^{(1)}\colon C_0(0,1] \otimes E \rightarrow \overline{h_nJh_n}$. 

Applying Proposition~\ref{prop:NDzeroFull*hom} twice, we obtain c.p.c.\ maps 
\begin{equation}
    \begin{split}  
        \widehat{\eta}_n^{(0)}&\colon C_0(0,1] \otimes E \rightarrow \M_{r_n^{(0)}}, \quad
        \xi_n^{(0)}\colon \M_{r_n^{(0)}} \rightarrow \M_2(\overline{h_{n-1}Jh_{n-1}}),\\
        \widehat{\eta}_n^{(1)}&\colon C_0(0,1] \otimes E \rightarrow \M_{r_n^{(1)}}, \quad
        \xi_n^{(1)}\colon \M_{r_n^{(1)}} \rightarrow \M_2(\overline{h_{n}Jh_{n}})
    \end{split}
\end{equation}
such that $\xi_n^{(0)}$ and $\xi_n^{(1)}$ are $^*$-homomorphisms and
\begin{align}
     \lim_{n\to\omega} \| \iota_{11} \circ \widehat{\alpha}_n^{(i)}(f) - \xi_n^{(i)} \circ \widehat{\eta}^{(i)}_n(f) \| = 0 \label{eqn:almost-there}
\end{align}
for $i \in \{0,1\}$ and $f \in C_0(0,1] \otimes E$. Define $\eta_n^{(i)}\colon E \rightarrow \M_{r_n^{(i)}}$ by $\eta_n^{(i)}(a) = \widehat{\eta}_n^{(i)}(\id_{(0,1]} \otimes a)$ for $i \in \{0,1\}$ and $a \in E$. 

Let $a \in E$. Since
\begin{equation}
    \begin{split}   
    \alpha(a) &= \alpha^{(0)}(a) + \alpha^{(1)}(a)\\ &= \HM(\alpha^{(0)})(\id_{(0,1]} \otimes a) + \HM(\alpha^{(1)})(\id_{(0,1]} \otimes a),
    \end{split}
\end{equation}
it follows from \eqref{eqn:almost-there} that
   \begin{equation}
       \lim_{n\to\omega} \| \iota_{11} \circ \alpha_n(a) - \xi^{(0)}_n \circ \eta^{(0)}_n (a) - \xi^{(1)}_n \circ \eta^{(1)}_n (a)\| = 0.\qedhere
   \end{equation}
\end{proof}

\section{Proof of Theorem~\ref{thm:main}}\label{sec:mainthm}

With the technical machinery in place, we are now ready to proceed with the proof of the main theorem as outlined in Section \ref{sec:outline}.

\begin{proof}[Proof of Theorem~\ref{thm:main}]
By \cite[Proposition~2.3]{WZ10}, we have 
\begin{equation}
    \dimnuc(E) \geq \max(\dimnuc(J), \dimnuc(C(X))).
\end{equation}
By \cite[Proposition~2.4]{WZ10}, we have $\dimnuc(C(X)) = \dim(X)$. 
As $J$ is Kirchberg, $\dimnuc(J) = 1$ by \cite[Theorem~G]{BBSTWW}. Hence, $\dimnuc(E) \geq \max(1,\dim(X))$. It remains to prove the reverse inequality.

By Proposition~\ref{prop:beta-full}, there exist a c.p.c.\ splitting $\mu\colon C(X) \rightarrow E$ and an almost idempotent quasicentral approximate unit $(h_n)_{n=1}^\infty$ such that the corresponding Brake--Winter decomposition (see Proposition~\ref{prop:BW}) has the property that $\HM(\beta)$ is injective. 

Resume the notation established in Proposition~\ref{prop:BW}.
Define the three families of hereditary subalgebras $A_n = \overline{h_nJh_n}$, 
$B_n = \overline{(h_{n+1} - h_n)J(h_{n+1} - h_n)}$ 
and $C_n = \overline{(1 - h_{n+1})E(1 - h_{n+1})}$,
and sequences of c.p.c.\ maps 
\begin{equation}
    \begin{split}
\alpha_n\colon  E \rightarrow A_n,\quad  a &\mapsto h_n^{1/2} a h_n^{1/2},\\
\beta_n\colon   C(X)\rightarrow B_n,\quad f &\mapsto (h_{n+1} - h_n)^{1/2}\mu(f)(h_{n+1} - h_n)^{1/2},\\
\gamma_n\colon  C(X) \rightarrow C_n,\quad f &\mapsto (1-h_{n+1})^{1/2} \mu(f) (1-h_{n+1})^{1/2}.
    \end{split}
\end{equation}
Let $\alpha\colon E \rightarrow \prod_{n\to\omega} A_n \subseteq E_\omega$, $\beta\colon C(X) \rightarrow \prod_{n\to\omega} B_n \subseteq E_\omega$ and $\gamma\colon C(X) \rightarrow \prod_{n\to\omega} C_n \subseteq E_\omega$ be the induced c.p.c.\ maps into the ultraproducts. Proposition~\ref{prop:BW} ensures that $\alpha$, $\beta$ and $\gamma$ are c.p.c.\ order zero maps, and we have 
    \begin{equation}\label{eqn:BW-approx}
	\alpha(a) + \beta (\pi(a)) + \gamma(\pi(a)) = a
\end{equation}
for all $a \in E$. Write $h = [h_n]_{n=1}^\infty \in E_\omega$ and $\hat{h} = [h_{n+1}]_{n=1}^\infty \in E_\omega$.

Let $d = \dim(X)$. There's nothing to prove when $d$ is infinite, so we may assume $d$ is finite. Fix $k \in \N$.
By \cite[Lemma~4.1]{GT22}, there exist a finite open cover $\mathcal{U}_k = \{U^{(i)}_{k,j}\colon  i=0,\ldots,d; j=1,\ldots,r(k,i)\}$ of $X$ and a
finite Borel partition $\mathcal{Y}_k = \{Y_{k,j}\colon  j=1,\ldots,r(k,0)\}$ of $X$ such that:
\begin{enumerate}[(i)]
  \item $U^{(i)}_{k,j} \cap U^{(i)}_{k',j'} = \emptyset$ whenever $j \neq j'$ and for all $r=0,\dots, d$;\label{MainThm.item1}
  \item each $U^{(i)}_{k,j}$ and each $Y_{k,j}$ has diameter at most $2^{-k}$;\label{MainThm.item2}
  \item $U_{k,j}^{(0)} \subseteq Y_{k,j}$ for $j=1,\ldots,r(k,0)$; and\label{MainThm.item3}
  \item $U^{(i)}_{k,j} \neq \emptyset$ for $i \neq 0$.\label{MainThm.item4}
\end{enumerate}

Let $(g^{(i)}_{k,j})_{i,j}$ be a continuous partition of unity for $C(X)$ subordinate to $\mathcal{U}_k$. Note, if $U^{(0)}_{k,j} = \emptyset$ then $g^{(i)}_{k,j} = 0$.
For each triple $(i,j,k)$ choose a point $t^{(i)}_{k,j} \in U^{(i)}_{k,j}$ unless $U^{(0)}_{k,j} = \emptyset$ in which case chose $t^{(0)}_{k,j} \in Y_{k,j}$ 

Define $\psi_k = (\psi^{(i)}_{k})_{i=0}^d\colon C(X) \rightarrow \C^{r(k,0)} \oplus \cdots \oplus \C^{r(k,d)} $ by
\begin{equation}
	\psi^{(i)}_{k}(f) = (f(t^{(i)}_{k,1}), \ldots, f(t^{(i)}_{k,r(k,i)})). 
\end{equation} 
Define maps $\phi^{(i)}_{k}\colon \C^{r(k,i)} \rightarrow C(X)$ by
\begin{equation}
	\phi^{(i)}_{k}(\lambda^{(i)}_{1},\ldots,\lambda^{(i)}_{r(k,i)}) = \sum_{j=1}^{r(k,i)}\lambda^{(i)}_j g^{(i)}_{k,j}
\end{equation} 
and let $\phi_k\colon \mathbb C^{r(k,0)}\oplus\cdots\oplus \mathbb C^{r(k,d)}\to C(X)$ be the sum of these maps:
\begin{equation}
\phi_k((\lambda^{(i)}_j)_{j=1}^{r(k,0)},\cdots,(\lambda^{(i)}_j)_{j=1}^{r(k,d)})=\sum_{i=0}^d\phi_k^{(i)}(\lambda_1^{(i)},\dots,\lambda_{r(k,i)}^{(i)}).
\end{equation}
Let $B(X)$ be the $\mathrm{C}^*$-algebra of bounded Borel functions $X \rightarrow \C$. Define $\rho_k \colon \C^{r(k,0)} \rightarrow B(X)$ by
\begin{equation}
	\rho_{k}(\lambda^{(0)}_{1},\ldots,\lambda^{(0)}_{r(k,0)}) = \sum_{j=1}^{r(k,0)}\lambda^{(0)}_j \chi_{Y_{k,j}}.
\end{equation}
Note that $\psi_k$ and $\rho_k$ are $^*$-homomorphisms and (from the disjointness condition in \eqref{MainThm.item1}) each $\phi^{(i)}_{k}$ is a c.p.c.\ order zero map.
Since continuous functions on $X$ are uniformly continuous, condition~\eqref{MainThm.item2} gives
\begin{equation}\label{eqn:C(X)-approx}
 \lim_{k\to\infty} \phi_k \circ \psi_k(f) = f = \lim_{k\to\infty} \rho_k \circ \psi^{(0)}_k(f) \quad \text{ for all $f \in C(X)$.}
\end{equation}

Let $D \subseteq B(X)$ be the separable subalgebra generated by   
$C(X)$ together with the characteristic functions $\chi_{Y_{k,j}}$ for all $k \in \N$ and $j=1,\ldots,r(k,0)$.  By construction, the image of $\rho_k$ is contained in $D$ for all $k \in \N$.

Abusing notation slightly, we write $\iota_{11}$ for the $(1,1)$-corner embedding $\iota_{11}\colon E \rightarrow \M_2(E)$ and the induced map $E_\omega \rightarrow \M_2(E_\omega)$ as well as its restrictions to maps
$A_\omega \rightarrow \M_2(A_\omega)$, $B_\omega \rightarrow \M_2(B_\omega)$, and $C_\omega \rightarrow \M_2(C_\omega)$. 

Since $\HM(\beta)$ is injective, $\iota_{11} \circ \beta\colon C(X) \rightarrow \M_2(B_\omega)$ can be extended to a c.p.c.\ order zero map $\hat{\beta}\colon D \rightarrow \M_2(B_\omega)$ by Proposition~\ref{prop:extend}. 
We claim that $\hat{\beta} \circ \rho_k + \iota_{11} \circ \gamma \circ \phi^{(0)}_k$ is c.p.c.\ and order zero. Indeed, let  $e_j$,$e_{j'} \in \C^{r(k,0)}$ be distinct elements of the standard basis, and let $(\sigma_m)_{m=1}^\infty$ be a sequence of functions in $C(X)$ supported on $U^{(0)}_{k,j'}$ with $\lim_{m\to\infty} \|\sigma_mg^{(0)}_{k,j'} - g^{(0)}_{k,j'}\| = 0$. Since $\operatorname{supp}(\sigma_m) \cap Y_{k,j} = \emptyset$ (by condition \eqref{MainThm.item3}), we have $\chi_{Y_{k,j}} \leq 1-\sigma_m$. Hence,  
\begin{equation}\label{eqn:oz-test-1}
    \begin{split}
 \|(\hat{\beta} \circ \rho_k)(e_j)(\iota_{11} \circ \gamma \circ \phi^{(0)}_k)(e_{j'})\| &= \|\hat{\beta}(\chi_{Y_{k,j}})\gamma(g^{(0)}_{k,j'})\|\\
 &\leq  \|\beta(1-\sigma_m)\gamma(g^{(0)}_{k,j'})\|.
    \end{split}
\end{equation}
Using repeatedly that $h,\hat{h} \in E_\omega \cap E'$ and using, in the last line, that $\{\mu(f_1)\mu(f_2) - \mu(f_1f_2) : f_1, f_2 \in C(X)\} \subseteq J$, so is annihilated by $(\hat{h}-h)$, we compute: 
\begin{equation}\label{eqn:oz-test-2}
    \begin{split}
  \beta(1-\sigma_m)\gamma(g^{(0)}_{k,j'}) &= (\hat{h}-h)\mu(1-\sigma_m)(1-\hat{h})\mu(g^{(0)}_{k,j'})\\
 &= (\hat{h}-h)\mu(1-\sigma_m)\mu(g^{(0)}_{k,j'})(1-\hat{h})\\
 &= (\hat{h}-h)\mu((1-\sigma_m)g^{(0)}_{k,j'})(1-\hat{h}).
    \end{split}
\end{equation}
Since $(1-\sigma_m)g^{(0)}_{k,j'} \to 0$ as $m \to \infty$, it follows from \eqref{eqn:oz-test-1} and \eqref{eqn:oz-test-2} that 
\begin{equation}
    (\hat{\beta} \circ \rho_k)(e_j)(\iota_{11} \circ \gamma \circ \phi^{(0)}_k)(e_{j'}) = 0.
\end{equation}
Hence, $\hat{\beta} \circ \rho_k + \iota_{11} \circ \gamma \circ \phi^{(0)}_k$ is c.p.c.\ order zero. 

The image of $\hat{\beta} \circ \rho_k + \iota_{11} \circ \gamma \circ \phi^{(0)}_k$ is contained in the hereditary subalgebra $\prod_{n\to\omega}\M_2(\overline{(1-h_n)E(1-h_n)})$
of $\M_2(E_\omega)$. Since the domain $\C^{r(k,0)}$ of the order-zero map $\hat{\beta} \circ \rho_k + \iota_{11} \circ \gamma \circ \phi^{(0)}_k$ is a finite-dimensional abelian $\mathrm{C}^*$-algebra, \cite[Lemma~4.6]{Lo93} (which reinterprets \cite[Lemma~2.6]{AkPed77})\footnote{An explicit statement of this result, as well as the more general liftability of c.p.c.\ order zero maps with finite-dimensional domains (obtained by reinterpreting results from \cite{Lo93}), can be found as \cite[Proposition~1.2.4]{Wi09}.} implies that
we can lift $\hat{\beta} \circ \rho_k + \iota_{11} \circ \gamma \circ \phi^{(0)}_k$ to a sequence of order zero maps $\hat{\phi}^{(0)}_{k,n}\colon \C^{r(k,0)} \rightarrow \M_2(\overline{(1-h_n)E(1-h_n)})$. Similarly, we can lift $\iota_{11} \circ \gamma \circ \phi^{(i)}_k$ to a sequence of order zero maps $\hat{\phi}^{(i)}_{k,n}\colon \C^{r(k,i)} \rightarrow \M_2(C_n)$ for each $i > 0$. 

Since $C(X)$ and $J$ are separable and nuclear, $E$ is separable and is nuclear by \cite[Corollary~3.3]{CE76}.
By Proposition~\ref{prop:Approximate-Alpha}, there exist c.p.c.\ maps $\eta_n^{(0)}\colon E \rightarrow \M_{s_n^{(0)}}$ and $\eta_n^{(1)}\colon E \rightarrow \M_{s_n^{(1)}}$, and $^*$-homomorphisms $\xi_n^{(0)}\colon \M_{s_n^{(0)}} \rightarrow  \M_2(A_{n-1})$ and $\xi_n^{(1)}\colon \M_{s_n^{(1)}} \rightarrow  \M_2(A_{n})$ such that 
\begin{equation}\label{eqn:AnAprproximation}
	\lim_{n\rightarrow\omega} \|\iota_{11}(\alpha_n(a)) - \xi_n^{(0)}(\eta_n^{(0)}(a)) - \xi_n^{(1)}(\eta_n^{(1)}(a)) \| = 0
\end{equation}
for all $a \in E$. 

The image of $\xi_n^{(0)}$ lies in $\M_2(A_{n-1})$, which is orthogonal to the hereditary subalgebra $\M_2(\overline{(1-h_n)E(1-h_n)})$. 
Hence, $\xi_n^{(0)} + \hat{\phi}^{(0)}_{k,n}$ is a c.p.c.\ order zero map $\M_{s_n^{(0)}} \oplus \C^{r(k,0)} \rightarrow \M_2(E_\omega)$.
The image of $\xi_n^{(1)}$ lies in $\M_2(A_{n})$, which is orthogonal to $\M_2(C_n)$. Hence, $\xi_n^{(1)} + \hat{\phi}^{(1)}_{k,n}$ is also c.p.c.\ order zero.

We define an approximation $\Theta_{k,n}$ to the embedding $\iota_{11}\colon E \rightarrow \M_2(E)$ that factors through the finite-dimensional $\mathrm{C}^*$-algebras $(\M_{s_n^{(i)}} \oplus \C^{r(k,i)})$ as follows: the downward c.p.c.\ maps are
\begin{equation}
\begin{split}
	 \eta_n^{(0)}\oplus (\psi_k^{(0)} \circ \pi)  &\colon E \rightarrow \M_{s_n^{(0)}}(\C) \oplus \C^{r(k,0)},\\ 
	 \eta_n^{(1)}\oplus (\psi_k^{(1)} \circ \pi)  &\colon E \rightarrow \M_{s_n^{(1)}}(\C) \oplus \C^{r(k,1)} ,\\
	\psi_k^{(i)} \circ \pi &\colon  E \rightarrow \C^{r(k,i)}, \quad\quad i=2,\ldots,d;
\end{split}
\end{equation}
the upward c.p.c.\ order zero maps are
\begin{equation}
\begin{split}
	\xi_n^{(0)} + \hat{\phi}^{(0)}_{k,n} &\colon\M_{s_n^{(0)}}(\C) \oplus \C^{r(k,0)} \rightarrow \M_2(E),\\
	  \xi_n^{(1)} + \hat{\phi}^{(1)}_{k,n} &\colon\M_{s_n^{(1)}}(\C) \oplus \C^{r(k,1)} \rightarrow \M_2(E),\\
	\hat{\phi}^{(i)}_{k,n}&\colon \C^{r(k,i)} \rightarrow \M_2(E), \quad\quad i=2,\ldots,d;
\end{split}
\end{equation}
and we set
\begin{equation}
\begin{split}
    \Theta_{k,n}(a) &=  \sum_{i=0}^1 (\xi_n^{(i)}(\eta_n^{(i)}(a)) + \hat{\phi}^{(i)}_{k,n}(\psi_k^{(i)}(\pi(a))))\\
    &\quad \quad  \quad
    + \sum_{i=2}^d \hat{\phi}^{(i)}_{k,n}(\psi_k^{(i)}(\pi(a)))\\
    &= \sum_{i=0}^1 \xi_n^{(i)}(\eta_n^{(i)}(a))
    + \sum_{i=0}^d \hat{\phi}^{(i)}_{k,n}(\psi_k^{(i)}(\pi(a)))
\end{split}
\end{equation}
for all $a \in E$. (If $d \leq 1$, there are only two downward and two upward maps. If $d=0$, then $\psi_k^{(1)}$ and $\hat{\phi}^{(1)}_{k,n}$ should be interpreted as zero maps.)

Let $a \in E$ and $k \in \N$. Using~\eqref{eqn:AnAprproximation} at the first equality, and the definition of $\hat{\phi}^{(0)}_{k,n}$ and $\hat{\phi}^{(i)}_{k,n}$ for $i=1,\ldots,d$ at the second, we compute in the ultrapower $E_\omega$:
\begin{equation}
\begin{split}
    [\Theta_{k,n}(a)]_{n=1}^\infty &= \iota_{11}(\alpha(a)) + \sum_{i=0}^d [\hat{\phi}^{(i)}_{k,n}(\psi_k^{(i)}(\pi(a)))]_{n=1}^\infty\\
    &= \iota_{11}(\alpha(a)) \\
    &\quad \quad + (\hat{\beta} \circ \rho_k + \iota_{11} \circ \gamma \circ \phi^{(0)}_k)(\psi_k^{(0)}(\pi(a)))\\
    &\quad \quad + \sum_{i=1}^d \iota_{11} \circ \gamma \circ \phi^{(i)}_k(\psi_k^{(i)}(\pi(a)))\\
    &= \iota_{11}(\alpha(a)) + \hat{\beta}(\rho_k \circ \psi_k^{(0)}(\pi(a)))\\
    &\quad \quad + \sum_{i=0}^d \iota_{11}(\gamma(\phi^{(i)}_k \circ \psi_k^{(i)}(\pi(a)))).
\end{split}
\end{equation}
Now, using~\eqref{eqn:C(X)-approx} in the first line, $\hat{\beta}|_{C(X)} = \iota_{11}\circ\beta$ in the second, and ~\eqref{eqn:BW-approx} in the third, we obtain
\begin{equation}
\begin{split}
    \lim_{k\to\infty}[\Theta_{k,n}(a)]_{n=1}^\infty &= \iota_{11}(\alpha(a)) + \hat{\beta}(\pi(a)) + \iota_{11}(\gamma(\pi(a)))\label{eqn:Theta-Ultrapower}\\
    &= \iota_{11}(\alpha(a)) + \iota_{11}(\beta(\pi(a))) + \iota_{11}(\gamma(\pi(a)))\\
    &= \iota_{11}(a).
\end{split}
\end{equation}
Reformulating \eqref{eqn:Theta-Ultrapower}, we have
\begin{equation}
 \lim_{k\to\infty}\lim_{n\to\omega} \|\iota_{11}(a) - \Theta_{k,n}(a)\| = 0
\end{equation} 
for all $a \in E$.
Therefore, by Definition~\ref{def:DimNuc}, the nuclear dimension of $\iota_{11}\colon E \rightarrow \M_2(E)$ is at most $\max(1,d)$.

Identify $E$ with the hereditary subalgebra of $\M_2(E)$ consisting of matrices that are only non-zero in the (1,1) entry. 
Then $\mathrm{id}_E$ is the co-restriction of $\iota_{11}$ to $E$.
Hence, by \cite[Proposition~1.6]{BGSW19}, we have $\dimnuc(E) \leq \max(1,d)$. 
\end{proof}

We strengthen Theorem~\ref{thm:main} a posteriori to allow for the quotient appearing in the exact sequence~\eqref{eqn:extension1} just to be stably isomorphic to $C(X)$. We need a simple stabilisation lemma.

\begin{lemma}\label{lem:extension stabilisation}
Let $B$ be a unital $\mathrm{C}^*$-algebra, let $A$ be a $\mathrm{C}^*$-algebra stably isomorphic to $B$, let $J$ be a stable $\sigma$-unital $\mathrm{C}^*$-algebra, and let 
\begin{equation}\label{eqn:general extension by stable}
0 \to J \to E \to A \to 0 
\end{equation}
be an extension. Then there is an extension $0 \to J \to F \to B \to 0$ such that $F \otimes \K \cong E \otimes \K$. If~\eqref{eqn:general extension by stable} is essential, then the extension $0 \to J \to F \to B \to 0$ can be chosen also to be essential.
\end{lemma}
\begin{proof}
Since $J$ is stable, by tensoring~\eqref{eqn:general extension by stable} throughout by $\K$, we can assume that $A = B \otimes \K$. 
Fix a positive lift $f \in E$ of the projection $p = 1_B \otimes e_{11} \in B \otimes \K$.
Let $F = J + \overline{fEf}$.  By construction $J \triangleleft F$ and the quotient is isomorphic to $(1_B \otimes e_{11})(B \otimes \K)(1_B \otimes e_{11}) \cong B$. If $J \triangleleft F$ is not essential, there is a non-zero $x \in F$ such that $x J = 0$, which implies that $J$ is not essential in $E$.

By hypothesis, $J$ is $\sigma$-unital, and $A = B \otimes \K$ is $\sigma$-unital as $B$ is unital.
Hence, there exist $h \in J_+$ and $a \in A_+$ with $J = \overline{hJh}$ and $A = \overline{aAa}$.
Let $e \in E_+$ be a lift of $a$.
Then $E = \overline{(h+e)E(h+e)}$ and $F = \overline{(h+f)F(h+f)}$, so $E$ and $F$ are $\sigma$-unital.
Moreover, since $F = \overline{(h+f)E(h+f)}$, it is a hereditary subalgebra of $E$, and it is full because it contains $J$ and $\pi(f)$ generates $A$ as an ideal.
Therefore, \cite[Theorem~2.8]{Br77} implies that $E \otimes \K \cong F \otimes \K$.
\end{proof}

\begin{corollary}\label{cor:stable quotient}
Let $J$ be a stable Kirchberg algebra and $X$ a compact metric space. Suppose that $A$ is a $\mathrm{C}^*$-algebra that is stably isomorphic to $C(X)$, and let
\begin{equation}\label{eqn:extension by stable}
 0 \rightarrow J \rightarrow E \xrightarrow{\pi} A \rightarrow 0
\end{equation}
be an essential extension. Then $\dimnuc{E} = \max(1,\dim(X))$.
\end{corollary}
\begin{proof}
Nuclear dimension is preserved by stable isomorphism \cite[Corollary~2.8(i)]{WZ10} and $A$ and $J$ are $\sigma$-unital, so by Lemma~\ref{lem:extension stabilisation} we can replace $A$ with $C(X)$; and then Theorem~\ref{thm:main} applies.
\end{proof}

\section{Graph algebras}\label{sec:graphs}

In this section we relate our results to graph $\mathrm{C}^*$-algebras, describing exactly which finite graphs have $\mathrm{C}^*$-algebras that satisfy the hypotheses of Theorem~\ref{thm:main} (or rather, Corollary~\ref{cor:stable quotient}).

If $E$ is a directed graph, then its vertex and edge sets are denoted $E^0$ and $E^1$ respectively. The range and source maps $r, s: E^1 \to E^0$ encode the directions of the edges. For $n > 1$, we write $E^n = \{e_1\cdots e_n : e_i \in E^1, r(e_i) = s(e_{i+1})\}$ for the set of paths of length $n$, and for $\mu = e_1\cdots e_n \in E^n$ we write $s(\mu) = s(e_1)$ and $r(\mu) = r(e_n)$. We write $E^* = \bigcup_{n \ge 0} E^n$. Note that our edge direction conventions are as per the earlier papers on graph algebras, as these form the primary sources for almost all the results we cite. Today it is more typical to use the opposite direction conventions.

A \emph{sink} in $E$ is a vertex $v \in E^0$ such that $s^{-1}(v) = \emptyset$. So we say that $E$ \emph{has no sinks} if $s : E^1 \to E^0$ is surjective. We say that $E$ is \emph{connected} if the smallest equivalence relation $\sim$  on $E^0$ such that $r(e) \sim s(e)$ for all $e \in E^1$ is the complete equivalence relation. We say that $E$ is \emph{strongly connected} if for all $v,w \in E^0$ there exists $\mu \in E^*$ such that $s(\mu) = v$ and $r(\mu) = w$. (With these conventions, the graph with one vertex and no edges is strongly connected and has a sink.)

A subset $V \subseteq E^0$ is \emph{hereditary} if whenever $e \in E^1$ satisfies $s(e) \in V$, we also have $r(e) \in V$. Equivalently, $V$ is hereditary if whenever $\mu \in E^*$ satisfies $s(\mu) \in V$ we also have $r(\mu) \in V$.

A subset $V \subseteq E^0$ is \emph{saturated} if $v \in V$ whenever $v \in E^0$ has the property that $s^{-1}(v) \neq \emptyset$ and $r(e) \in V$ for every $e \in s^{-1}(v)$. 
The \emph{saturation} $\overline{V}$ of a subset $V \subseteq E^0$ is the smallest saturated set containing $V$. If $E$ is finite and has no sinks, as is the case for the graphs in this paper, then an induction on the length of paths shows that
\begin{equation}\label{Graph.NewEq}
 \{v \in E^0 : \exists n \ge 0\text{ s.t. } \forall \mu\in E^n, \big(s(\mu) = v \implies r(\mu) \in V\big)\}\subseteq \overline{V}.
\end{equation}
If, additionally, $V$ is hereditary, then we have equality in \eqref{Graph.NewEq}. If $V$ is hereditary, then $\overline{V}$ is also hereditary. 

Given any subset $V \subseteq E^0$, the \emph{subgraph of $E$ over $V$} is the directed graph with $E^0_V = V$ and $E^1_V = \{e \in E^1 : s(e) \in V \text{ and } r(e) \in V\}$, and with source and range maps inherited from $E$. If $V$ is a hereditary set, then $E^1_V$ is precisely $s^{-1}(V)$.

A \emph{cycle} in $E$ is a path $\mu = e_1\cdots e_n \in E^* \setminus E^0$ such that $s(\mu) = r(\mu)$. An \emph{exit} from this cycle is an edge $e \in E^1$ such that $s(e) = s(e_i)$ but $e \not= e_i$ for some $i \le n$. For $C \subseteq E^0$, we say that $E_C$ is a \emph{simple cycle} if there is a cycle $\mu = e_1 \dots e_{|C|}$ in $E_C$ with $e_i \not= e_j$ for all $i\not= j$ such that $E_C^1 = \{e_1, \dots, e_{|C|}\}$. This forces $E^1_C \not= \emptyset$, so the graph with one vertex and no edges is not a simple cycle.

The \emph{graph $\mathrm{C}^*$-algebra} $\mathrm{C}^*(E)$ is the universal $\mathrm{C}^*$-algebra generated by mutually orthogonal projections $\{p_v : v \in E^0\}$ and partial isometries $\{s_e : e \in E^1\}$ with mutually orthogonal range projections such that $s_e^* s_e = p_{r(e)}$ and $s_es_e^* \leq p_{s(e)}$ for all $e \in E^1$, and $p_v = \sum_{s(e) = v} s_e s_e^*$ whenever $0 < |s^{-1}(v)| < \infty$. For $V \subseteq E^0$, we write $I_V$ for the ideal of $\mathrm{C}^*(E)$ generated by $\{p_v : v \in V\}$. 

For $\mu = e_1 \cdots e_n \in E^*$. we write $s_\mu = s_{e_1} \cdots s_{e_n}$. For any distinct $e,e' \in E^1$, the range projections of $s_e$ and $s_{e'}$ are orthogonal, so $s_e^*s_{e'} = 0$. Therefore, for any $\mu,\nu\in E^*$, the projections $s_\mu s_\mu^*$ and $s_\nu s_\nu^*$ are orthogonal unless either $\mu$ factorises as $\nu\mu'$ or $\nu$ factorises as $\mu\nu'$.

If $\mu$ is a cycle with an exit, say $s(e) = s(e_i)$ but $e \not= e_i$, then $\nu = e_1 \cdots e_{i-1} e$ satisfies $s_\mu s^*_\mu \le p_{s(\mu)} - s_\nu s^*_\nu < p_{r(\mu)} = s_\mu^*s_\mu$. So any $\mathrm{C}^*$-subalgebra of $\mathrm{C}^*(E)$ containing $s_\mu$ has an infinite projection.  The graph $\mathrm{C}^*$-algebra of a strongly connected finite directed graph $E$ that has no sinks and is not a simple cycle is a Kirchberg algebra. In fact, by the work of \cite{KPRR97,KPR98} it is a simple Cuntz--Krieger algebra (see \cite[Pages 78--80]{Rordam2002}).

We now state our characterisation of the finite graphs whose $\mathrm{C}^*$-algebras satisfy the hypotheses of Corollary~\ref{cor:stable quotient}. 

\begin{proposition}\label{prop:which graph algebras}
Let $E$ be a finite directed graph. Then there is an exact sequence $0 \to J \to \mathrm{C}^*(E) \overset{\pi}{\to} A \to 0$ in which $J$ is essential and is a stable Kirchberg algebra and $A$ is stably isomorphic to $C(X)$ for some compact Hausdorff space $X$ if and only if there is a decomposition $E^0 = V_0 \sqcup V_1 \sqcup V_2$ such that
\begin{enumerate}
    \item the subgraph $E_0 = E_{V_0}$ over $V_0$ is the saturation of a nonempty disjoint union $\bigsqcup_{C \in \mathcal{C}} C$ such that each $E_C$ is a simple cycle, 
    \item the subgraph $E_1 = E_{V_1}$ over $V_1$ contains no cycles,
    \item the subgraph $E_2 = E_{V_2}$ over $V_2$ is strongly connected, has at least one edge, and is not a simple cycle,
    \item whenever $i > j$, there are no paths $\mu \in E^*$ such that $r(\mu) \in V_j$ and $s(\mu) \in V_i$, and
    \item for each $v \in V_0 \cup V_1$, there is a path $\mu \in E^*$ such that $s(\mu) = v$ and $r(\mu) \in V_2$. 
\end{enumerate}
In this case, $X \cong \bigsqcup_{C \in \mathcal{C}} \mathbb{T}$ is a disjoint union of circles and $J \cong \mathrm{C}^*(E_2) \otimes \K$.
\end{proposition}

To prove Proposition~\ref{prop:which graph algebras}, we will twice use the following folklore lemma.

\begin{lemma}\label{lem:graph circle algebras}
Let $E$ be a finite directed graph with no sinks. Suppose that $\mathcal{C}$ is a finite set of mutually disjoint subsets of $E^0$ such that for each $C \in \mathcal{C}$, $E_C$ is a simple cycle. Let $V = \bigcup_{C \in \mathcal{C}} C$. Then $I_V$ is stably isomorphic to $\bigoplus_{C \in \mathcal{C}} M_{|C|}(C(\mathbb{T}))$. 
\end{lemma}
\begin{proof}
The ideal $I_V$ is generated as an ideal by $\{p_v\colon v\in C\in\mathcal C\}$, so applying \cite[Theorem~4.1(c)]{BPRS} with $X=\bigsqcup_{C\in \mathcal C}C$, it follows that $\bigoplus_{C \in \mathcal{C}} \mathrm{C}^*(E_C)$ is a full corner in $I_V$. Hence $I_V$ is stably isomorphic to $\bigoplus_{C \in \mathcal{C}} \mathrm{C}^*(E_C)$ by Brown's theorem (\cite[Theorem~2.8]{Br77}). Since each $E_C$ is a cycle with no exit, its edge-adjacency matrix is an irreducible permutation matrix of order $|C|$. Hence \cite[Theorem~2.2]{Evans82} implies that each $\mathrm{C}^*(E_C) \cong M_{|C|}(C(\mathbb{T}))$.
\end{proof}

\begin{proof}[Proof of Proposition~\ref{prop:which graph algebras}]
First suppose that $E^0 = V_0 \sqcup V_1 \sqcup V_2$ satisfies Conditions (1)--(5). 
Then $V_2$ is hereditary by Condition (4). 
Conditions (3)~and~(5) imply that $E$ has no sinks. Since $V_1$ is finite, Condition~(2) implies that there exists $n \ge 0$ such that $r(\mu) \not \in V_1$ for all $\mu \in E^n$ with $s(\mu) \in V_1$; and then Condition~(4) implies that $r(\mu) \in V_2$ for all $\mu \in E^n$ with $s(\mu) \in V_1$. Consequently, $V_1$ is contained in the saturation $\overline{V_2}$ of $V_2$. 
Condition~(1) implies that $V_0 \cap \overline{V_1 \cup V_2} = \emptyset$. 
So $V_1 \cup V_2 = \overline{V_1 \cup V_2} = \overline{V_2}$ is a saturated hereditary subset of $E^0$ whose complement is $V_0$. 
Let $J = I_{V_2}$, the ideal of $\mathrm{C}^*(E)$ generated by $\{p_v : v \in V_2\}$. 
By \cite[Lemma~4.3]{BPRS}, $J = I_{V_1 \cup V_2}$. Taking $X = V_2$ in \cite[Theorem~4.1(c)]{BPRS} we see that 
\begin{equation}
\mathrm{C}^*(\{s_e, p_v : v \in V_2, e \in s^{-1}(V_2)\}) \subseteq \mathrm{C}^*(E)
\end{equation} is canonically isomorphic to $\mathrm{C}^*(E_2)$,  and is a full corner of $J$. By Brown's theorem (\cite[Theorem~2.8]{Br77}), $J$ and $\mathrm{C}^*(E_2)$ are stably isomorphic. Thus, since $\mathrm{C}^*(E_2)$ is a Kirchberg algebra (by Condition (3)), so too is $J$.

Taking $H = V_1 \cup V_2$ in \cite[Theorem~4.1(b)]{BPRS}, we obtain an exact sequence
\begin{equation}
    0 \to J \to \mathrm{C}^*(E) \to \mathrm{C}^*(E_0) \to 0.
\end{equation}
Since $E_0$ is the saturation of a disjoint union $\bigsqcup_{C \in \mathcal{C}} C$ of sets of vertices on cycles with no exits, Lemma~\ref{lem:graph circle algebras} implies that, writing $X = \bigsqcup_{C \in \mathcal{C}} \mathbb{T}$, we have $\mathrm{C}^*(E_0)$ is stably isomorphic to $C(X)$.

It remains to show that is stable and essential.
For the former, we adapt the idea of \cite[Proposition~6.4]{ET09}. 
Let 
\begin{equation}
F = V_2 \cup \{\mu = e_1 \dots e_n \in E^* : r(\mu) \in V_2\text{ and }s(e_n) \not\in V_2\} \subseteq E^*.
\end{equation} Since $V_2$ is hereditary, if $e_1\cdots e_n \in F$ then $s(e_i) \not\in V_2$ for all $i$. 
Hence, if $\mu, \nu \in F$ are distinct, then $\nu$ does not factorise as $\mu\nu'$ and vice-versa, and consequently $s_\mu s^*_\mu s_\nu s^*_\nu = 0$.
Fix $v\in V_2$. Following arbitrary large repetitions of a simple cycle in $E_0$ (given by Condition (5)) by a path connecting it to $v$ (given by Condition (3)), it follows that $Fv\coloneqq \{\nu \in F : r(\nu) = v\}$ is infinite. 
For each $v \in V_2$, write $\lambda^v_1=v$ and fix an enumeration $(\lambda_i^v)_{i=2}^\infty$ of $Fv \setminus \{v\}$. Since $(s_\mu s^*_\mu)_{\mu \in F}$ are mutually orthogonal, the elements $(S_i)_{i=1}^\infty$ defined by  $S_i = \sum_{v \in V_2} s_{\lambda^v_i}$ are partial isometries with mutually orthogonal range projections, and common initial projection $\sum_{v \in V_2} p_v$, which is the identity element of $A = \mathrm{C}^*(\{s_e, p_v : v \in V_2, e \in s^{-1}(V_2)\})$.
Hence $\theta_{i,j} = S_i S^*_j$ defines a family of matrix units generating a copy of the compact operators.
Direct calculation shows that there is a homomorphism $\pi : A \otimes \mathbb{K} \to J$ such that $\pi(a \otimes \theta_{i,j}) =  S_i a S^*_j$ for all $a \in A$ and $i,j \in \mathbb{N}$, so its image $\pi(A \otimes \mathbb{K})$ is stable, and it suffices to show that $\pi(A \otimes \mathbb{K}) = J$. By \cite[Lemma~4.3]{BPRS}, $J = \overline{\operatorname{span}}\,\{s_\mu s^*_\nu : r(\mu) = r(\nu) \in V_2\}$, so it suffices to show that if $r(\mu) \in V_2$, then $s_\mu \in \pi(A \otimes \mathbb{K})$.
So fix $\mu \in E^*$ with $r(\mu) \in V_2$. If $s(\mu)\in V_2$, then $s_\mu=\pi(s_\mu\otimes\theta_{1,1})$ as $\lambda_1^v=v$ for all $v\in V_2$. Otherwise we can factorise $\mu=\mu'\mu''$ where $\mu'\in F$ and $s(\mu'')\in V_2$.  
Again $s_{\mu''} = \pi(s_{\mu''} \otimes \theta_{1,1})$, and since $\mu' \in F$, it is equal to $\lambda^{r(\mu')}_i$ for some $i$. Together this gives $s_{\mu'} = S_i p_{r(\mu')} S_1^* = \pi(p_{r(\mu')} \otimes \theta_{i,1})$. Hence $s_\mu = \pi(p_{r(\mu')} \otimes \theta_{i,1})\pi(s_{\mu''} \otimes \theta_{1,1}) \in \pi(A \otimes \mathbb{K})$. This completes the proof that $J$ is stable. Since we showed above that $J$ is stably isomorphic to $\mathrm{C}^*(E_2)$, it now follows that $J \cong \mathrm{C}^*(E_2) \otimes \K$.

To see that $J$ is essential, fix a non-zero ideal $I\triangleleft \mathrm{C}^*(E)$. Note that Conditions (2)~and~(5) ensure that $E$ has no sinks, and Conditions (3)~and~(5) ensure that every cycle in $E$ has an exit.  Accordingly the Cuntz--Krieger uniqueness theorem (\cite[Theorem 3.7]{KPR98}) ensures that $I$ contains $p_v$ for some $v\in E^0$ (otherwise the quotient map $\mathrm{C}^*(E)\to \mathrm{C}^*(E)/I$ would be an isomorphism). Since $\{w : p_w \in I\} $ is hereditary (\cite[Lemma 4.2]{BPRS}), Conditions~(3) and (5) ensure that there exists $w \in V_2$ with $p_w \in I$ and hence in $I \cap J$.  

Now suppose there is an essential extension $0 \to J \to \mathrm{C}^*(E) \overset{\pi}{\to} A \to 0$ with $J$ a stable Kirchberg algebra and $A$ stably isomorphic to $C(X)$ for some compact Hausdorff space $X$. Let $W = \{v \in E^0 : p_v \in J\}$, and let $V_0 = E^0 \setminus W$; then $W$ is saturated and hereditary by \cite[Lemma 4.2]{BPRS}. Let $V_1$ be the set of vertices in $W$ that lie on no cycles in $E_W$, and let $V_2 = W \setminus V_1$. This $V_1$ satisfies Condition~(2) by definition. Since $W$ is hereditary and is disjoint from $V_0$ there is no path $\mu \in E^*$ with $r(\mu) \in V_0$ and $s(\mu) \in W$, which gives Condition~(4) for $j = 0$ and $i \ge 1$. The ideal $I_W = \overline{\operatorname{span}} \, \mathrm{C}^*(E)\{p_w : w \in W\} \mathrm{C}^*(E)$ is contained in $J$. We claim that $W \not= \emptyset$. Suppose for contradiction that $W = \emptyset$. Then $J$ is an ideal of $\mathrm{C}^*(E)$ containing no vertex projections, and so in the notation of \cite[Theorem~5.1]{CS2017},\footnote{Warning: the edge-direction convention in \cite{CS2017} is opposite to that in this paper.} we have $J = J_{\emptyset, U}$ for some function assigning a proper open subset $U(C)$ of $\mathbb{T}$ to each cycle $C$ with no exit in $E^0 \setminus \emptyset = E^0$. Let $K$ be the set of vertices that lie on a cycle with no exit in $E$. Then \cite[Theorem~5.1]{CS2017} says that $J$ is generated as an ideal by elements of $I_K$, and hence $J \subseteq I_K$. Lemma~\ref{lem:graph circle algebras} implies that $I_K$ is stably finite, and this contradicts that $J$ is purely infinite. So $W$ is nonempty as claimed, and hence $I_W$ is a non-zero ideal of $J$. Since $J$ is simple it follows that $I_W = J$. By \cite[Theorem~4.1(c)]{BPRS}, the $\mathrm{C}^*$-algebra $\mathrm{C}^*(E_W)$ of the subgraph $E_W$ over $W$ is Morita equivalent to $J$ and hence a Kirchberg algebra. In particular, $\mathrm{C}^*(E_W)$ is not AF, so \cite[Theorem~2.4]{KPR98} shows that $V_2 \not= \emptyset$.

We claim that for every $v \in V_2$ and every $w \in W$ there exists $\mu \in E^*_W$ such that $r(\mu) = v$ and $s(\mu) = w$. To see this, fix $v \in V_2$. Let $K = W \setminus \{s(\nu) : \nu \in E^*_W \text{ and }r(\nu) = v\}$. We claim that $K$ is saturated and hereditary. To see that it is hereditary, suppose that $e\in E^1$ has $s(e) \in K$. Then for every $\nu \in E^*$ with $s(\nu) = r(e)$, as $W$ is hereditary, $r(e)\in W$, and $\nu\in E^*_W$. Then $r(\nu) = r(e\nu) \neq v$, by definition of $K$, and so $r(e) \in K$. To see that $K$ is saturated, suppose that $w\in E^0$ has $s^{-1}(w)\neq\emptyset$ and $r(e) \in K$ for all $e \in E^1$ with $s(e) = w$. Then $w\in W$ as $W$ is saturated. Now fix $\nu \in E^*$ with $s(\nu) = w$. Then $\nu\in E_W^*$ (as $W$ is hereditary) and $\nu = e \nu'$ for some $e \in E^1$ with $s(e) = w$. Hence $r(e) \in K$ (as $K$ is hereditary) and therefore $r(\nu) = r(\nu') \neq v$. So $w \in K$. That is $K$ is saturated and hereditary as claimed. By definition, $v \not\in K$, so $K \not= W$. Since $\mathrm{C}^*(E_W)$ is simple, \cite[Theorem~4.1(a)]{BPRS} implies that the only saturated hereditary subsets of $E_W^0$ are $\emptyset$ and $W$, so $K = \emptyset$, proving the claim.

The claim of the preceding paragraph applied to each pair $v, w \in V_2$ shows that $E_2$ is strongly connected. By \cite[Proposition~5.1]{BPRS} for $\mathrm{C}^*(E_W)$ (which applies as the previous paragraph shows $E_W$ has no sinks) every cycle in $E_W$ has an exit, so $E_2$ is not a simple cycle. So $V_2$ satisfies Condition~(3). The claim also shows that for each $v \in V_1$ there is a path $\mu \in E^*$ such that $s(\mu) = v$ and $r(\mu) \in V_2$, which verifies Condition~(5) for $v \in V_1$. Moreover, if $w \in W$ and there exists $\nu \in E^*$ with $s(\nu) \in V_2$ and $r(\nu) = w$, then $\nu \in E_W^*$ because $W$ is hereditary. So the claim applied with $v = s(\nu)$ gives a cycle $\mu\nu$ with source $w$, and hence $w \in V_2$. Hence, if $w \in V_1$, then there is no path in $E$ to $w$ from $V_2$. This verifies the remaining case of Condition~(4) (for $i = 2$ and $j = 1$).

Since $\mathrm{C}^*(E)$ is unital and $J$ is stable, $J \not= \mathrm{C}^*(E)$, so $V_0 \not= \emptyset$. By \cite[Theorem~4.1(b)]{BPRS}, we have $A = \mathrm{C}^*(E)/J \cong \mathrm{C}^*(E_0)$. Hence $\mathrm{C}^*(E_0)$ is unital and stably isomorphic to $C(X)$, so stably finite. Thus no cycle in $E_0$ has an exit. Let $V_0'$ be the set of vertices in $V_0$ that lie on a cycle; so $V_0'$ is a finite union $\bigsqcup_{C \in \mathcal{C}} C$ of subsets such that each $E_C$ is a simple cycle. To verify Condition~(1) we must show that $V_0$ is the saturation in $E_0$ of $V'_0$. To see this, fix $v \in V_0$, let $n = |V_0|+1$, and suppose that $\mu \in E_0^n$ satisfies $s(\mu) = v$. By the pigeonhole principle $\mu$ must contain a cycle as a subpath. But since no cycle in $E_0$ has an exit in $E_0$, we have $r(\mu)\in V_0'$, and hence $v$ belongs to the saturation in $E_0$ of $V'_0$. So $V_0$ satisfies Condition~(1).

It remains to check Condition~(5) for $v \in V_0$. Suppose for contradiction that $v \in V_0$, and there is no $\mu \in E^*$ with $s(\mu) = v$ and $r(\mu) \in V_2$. Since $V_0$ is the saturation of $\bigsqcup_{C \in \mathcal{C}} C$, there is a path $\nu$ in $E_0^*$ and an element $C \in \mathcal{C}$ with $s(\nu) = v$ and $r(\nu) \in C$. Since $E_C$ is strongly connected and there is no path from $v$ to $V_2$, there is no path from $C$ to $V_2$. We already know from Condition~(4) that there are no paths from $V_2$ to $C$, so if $\mu,\nu \in E^*$ satisfy $r(\mu) \in C$ and $r(\nu) \in V_2$, then $\mu \not= \nu\mu'$ for any $\mu'$ and $\nu \not= \mu\mu'$ for any $\mu'$. Hence $s_\mu s^*_\mu$ and $s_\nu s^*_\nu$ are orthogonal. Since $C$ is hereditary, \cite[Lemma~4.3]{BPRS} implies that the ideal $I_C$ is $I_C = \overline{\operatorname{span}}\{s_\mu s^*_{\mu'} : r(\mu) = r(\mu') \in C\}$ and $J =I_W= \overline{\operatorname{span}}\{s_\nu s^*_{\nu'} : r(\nu) = r(\nu') \in V_2\}$. Hence $I_C \perp J$, which contradicts that $J$ is essential. This completes the verification of Condition~(5).
\end{proof}

\begin{remark}
Every graph $\mathrm{C}^*$-algebra is a direct limit of $\mathrm{C}^*$-algebras of finite graphs \cite{JP02}, and the $\mathrm{C}^*$-algebras of finite graphs have finite composition series with ideals and subquotients that are either finite-dimensional, matrix algebras over $C(\mathbb{T})$ or Kirchberg algebras. All these building blocks have nuclear dimension at most~1. In addition, among graph $\mathrm{C}^*$-algebras, extensions of AF quotients by purely infinite ideals \cite{RST15}, of purely-infinite quotients by AF ideals \cite{Gla20}, and of  commutative quotients by finite-dimensional ideals \cite{BW19,GT22} have nuclear dimension equal to the maximum of that of the quotient and that of the ideal. Corollary~\ref{cor:stable quotient} adds extensions of stably commutative quotients by stable Kirchberg ideals to this list. 

Suppose now that $E$ is a finite graph with a single saturated hereditary subgraph $E_1$ and complementary subgraph $E_0$ (so $\mathrm{C}^*(E)$ has a single nontrivial gauge-invariant ideal). Routine arguments show that if $E_1$ has a cycle with no exit and $E_0$ has no cycles, then $\mathrm{C}^*(E)$ has the form $M_m(C(\mathbb{T})) \oplus M_n$, and that if both $E_0$ and $E_1$ have no cycles then $\mathrm{C}^*(E)$ has the form $M_m \oplus M_n$. So we obtain the following table of values for the nuclear dimension of $\mathrm{C}^*(E)$ (a question mark indicates that the value is unknown), depending on whether each of the two components contains a cycle with an exit, contains a cycle without an exit, or contains no cycles.
\begin{center}
\begin{tabu}{|[1.5pt]l|l|[1.5pt]c|c|c|[1.5pt]}
\tabucline[1.5pt]{-}
\multicolumn{2}{|[1.5pt]c|[1.5pt]}{\multirow{2}{*}{\raisebox{0pt}[21pt][6pt]{}$\dimnuc(\mathrm{C}^*(E))$}}& 
\multicolumn{3}{c|[1.5pt]}{$E_1$}\\ \cline{3-5}
\multicolumn{1}{|[1.5pt]c}{ }&\multicolumn{1}{c|[1.5pt]}{ }&cycle with& cycle with-& no cycle\\
\multicolumn{1}{|[1.5pt]c}{ }&\multicolumn{1}{c|[1.5pt]}{ }&exit&out exit& \\
\tabucline[1.5pt]{-}
\multirow{3}{*}{$E_0$} &\raisebox{0pt}[12pt][6pt]{}cycle with exit & \textbf{1} \cite{FS23} & \textbf{?} &  \textbf{1} \cite{FS23} \\[-1pt]\tabucline{2-5}
  &\raisebox{0pt}[12pt][6pt]{}cycle without exit & \textbf{1} (Cor~\ref{cor:stable quotient}) & \textbf{?} & \textbf{1} \cite{GT22} \\[-1pt]\tabucline{2-5}
  &\raisebox{0pt}[12pt][6pt]{}no cycle &  \textbf{1} \cite{RSS15} & \textbf{1} \cite{WZ10} & \textbf{0} \cite{WZ10} \\\tabucline[1.5pt]{-}
\end{tabu}
\end{center}

Minimal examples of the two unknown cases are illiustrated below.
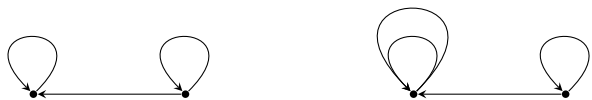
\[
\begin{tikzpicture}
\node[inner sep=1pt, fill=black, circle] (v) at (0,0) {};
\node[inner sep=1pt, fill=black, circle] (w) at (2,0) {};
\draw[-stealth] (w)--(v);
\draw[-stealth] (w) .. controls +(1, 1) and +(-1, 1) .. (w);
\draw[-stealth] (v) .. controls +(1, 1) and +(-1, 1) .. (v);
%
%
\node[inner sep=1pt, fill=black, circle] (x) at (5,0) {};
\node[inner sep=1pt, fill=black, circle] (y) at (7,0) {};
\draw[-stealth] (y)--(x);
\draw[-stealth] (x) .. controls +(1, 1) and +(-1, 1) .. (x);
\draw[-stealth] (x) .. controls +(1.5, 1.5) and +(-1.5, 1.5) .. (x);
\draw[-stealth] (y) .. controls +(1, 1) and +(-1, 1) .. (y);
\end{tikzpicture}
\]
We suspect that computations of nuclear dimension of $\mathrm{C}^*(E)$ for each of these two graphs $E$ would be a significant step towards breaking the back of Question~\ref{qn:dimnuc(graphalg)}.
\end{remark}

\paragraph*{\textbf{Acknowledgements}.  } The authors thank the anonymous referees for their careful reading of the paper and helpful comments.

\end{document}